\documentclass[11pt,a4paper,final]{article}
\usepackage{a4}
\usepackage{amsmath}%
\usepackage{amstext}%
\usepackage{amssymb}%
\usepackage{showkeys}%
\usepackage{cite}
\usepackage[latin1]{inputenc}
\usepackage{textcomp}
\usepackage{hyperref}
\usepackage{amsthm}
\usepackage{amsfonts}
\usepackage{graphicx}
\usepackage{xcolor}
\usepackage{tabularx}
\usepackage{fancyhdr}
\usepackage{multirow}

\usepackage{placeins}
\usepackage{pgfplots}
\usepackage{tikz}
\usepackage{epstopdf}
\usepackage{listings}
\usepackage[toc,page]{appendix}
\usepackage[refpage,intoc]{nomencl}
\usepackage{ifthen}
\usepackage{lipsum}
\usepackage{geometry}
\usepackage{booktabs}
\usepackage{caption}
\usepackage{subcaption}

\makenomenclature

\usepackage[margin=8pt,font=footnotesize,labelfont=bf,
labelsep=colon,format=hang,justification=justified]{caption}
\newtheorem{theorem}{Theorem}[section]
\newtheorem{remark}{Remark}[section]
\newtheorem{corollary}{Corollary}[section]

\newtheorem{example}{Example}[section]

\DeclareMathOperator*{\epi}{epi}
\DeclareMathOperator*{\sgn}{sgn}

\DeclareMathOperator*{\dom}{dom}

\DeclareMathOperator{\EVaR}{EV@R}
\DeclareMathOperator{\ess}{ess}

\DeclareMathOperator*{\interior}{int}
\DeclareMathOperator*{\ri}{ri}
\newcommand{\R}{\mathbb{R}}%

\def\smallskip{\vskip\smallskipamount}
\def\medskip{\vskip\medskipamount}
\def\bigskip{\vskip\bigskipamount}

\begin{document}

\title{A splitting proximal method in portfolio optimization with entropic value-at-risk}
\date{}

\author{{Nail Baloul\thanks{CERMICS, ENPC, Institut Polytechnique de Paris, CNRS, Marne-la-Vallée, France, e-mail: nail.baloul@enpc.fr
}}
\and {Sorin-Mihai Grad\thanks{Corresponding author. Faculty of Mathematics, University of Duisburg-Essen, Thea-Leymann-Str. 9, 45127 Essen, Germany, sorin-mihai.grad@uni-due.de, ORCID-ID: 0000-0002-1139-7504}
\thanks{Department of Mathematics, Babe\c s-Bolyai University Cluj-Napoca, Kog\u{a}lniceanu Street, No. 1, Cluj-Napoca, Romania
}}\and Oleg Wilfer\thanks{Faculty of Mathematics, Chemnitz University of Technology, D-09107 Chemnitz, Germany,
e-mail: oleg.wilfer@web.de}}

\maketitle

\begin{enumerate}
    \item[] \textbf{Abstract:} Motivated by a recent paper where a convex portfolio optimization problem formulated by means of the entropic value-at-risk is solved via a primal-dual interior-point method, we propose an alternative way of dealing with such problems by means of splitting proximal point methods and making use of Lagrange duality. Numerical experiments where the new method is implemented for solving some portfolio optimization problems are presented as well. As byproducts, we also provide formulae for the projection on the epigraph of the weighted Shannon entropy, and the conjugate function and subdifferential of entropic value-at-risk.

        \item[] \textbf{Key words:} entropic value-at-risk, portfolio optimization, Lagrange dual problem, proximal point method, epigraphical projection
\end{enumerate}

\section{Introduction and preliminaries}\label{se1}

In portfolio optimization (see \cite{SDL} for a recent survey on this already half-century old research area) one is interested in putting together a stock portfolio with maximal expected return and/or with minimal risk. There are different approaches to deal with this issue, usually by modelling it as an optimization problem. In the standard approaches, one either minimizes the risk while a minimal expected return is guaranteed, or the expected return can be maximized subject to a reasonable upper bounded risk. While the expected return is represented as a linear continuous function that does not produce many difficulties, the risk is modelled in various ways, mostly by means of so-called \textit{risk measures}. There are several classes of risk functions considered in the literature, among which we mention the coherent and the convex ones. Proposed a decade ago in \cite{AAJ}, the \textit{entropic value-at-risk} is a coherent risk measure that gained popularity in a short time being used by various authors in investigations in the field of mathematical finance or insurance theory (as it generalizes the exponential premium principle), but not only, see, for instance, \cite{AJM} for its employment in machine learning or the introduction of \cite{AJF} for other application areas. It is an upper bound for both value-at-risk and conditional value-at-risk, two of the most important risk measures previously considered in the literature. The entropic value-at-risk is strongly monotone over its domain and, according to \cite{AJF}, it also has better computational features than its mentioned better known counterparts, and more about it can be found also in \cite{AJP, AJF, AAJ, AJM, PIS}.

Motivated by the recent paper \cite{AJF}, where portfolio optimization problems under sample-based setting formulated by means of the entropic value-at-risk are investigated and solved with a (rather complicated) primal-dual interior-point method, we propose in this article an alternative approach for solving such optimization problems, where the involved functions lack differentiability, so classical gradient methods are not applicable. A direct usage of a (splitting) proximal point method for such a portfolio optimization problem is unfortunately not possible, too, due to the intricate structure of the latter. However, noticing a certain resemblance of the optimization problem considered in \cite{AJF} to the Fenchel-Lagrange dual problem obtained in \cite{BGW}, we investigate by means of conjugate duality a generalization of the primal entropy constrained linear optimization problem treated in the latter. After some transformations and simplifications, its Lagrange dual problem turns out to be the portfolio optimization problem from \cite{AJF}. Weak and strong duality for this primal-dual pair of optimization problems are provided as well as necessary and sufficient optimality conditions. The primal problem is then recast as an unconstrained optimization problem that is solved by means of a splitting proximal point method. The dual optimal solutions, i.e. the optimal solutions of the portfolio optimization problem, are then recovered via the optimality conditions. In order to employ the splitting method for solving the primal problem, new epigraphical projection results (complementing those from \cite{CPP, FGH}, among others) are provided. As byproducts, we also determine the conjugate function and the subdifferential of entropic value-at-risk. Numerical experiments where the new method is implemented for solving concrete portfolio optimization problems are presented as well. They reveal that the proposed approach is competitive with respect to the one from \cite{AJF} in solving such portfolio optimization problems.

Unlike other methods employed for solving similar optimization problems, the one based on the Parallel Splitting Algorithm considered in our paper can make use of parallel computing capabilities due to the fact that the proximal points of the involved functions can be determined separately and they are used together only in the last operation of an iterative cycle. Such an algorithm also requires way less storage capacity than the primal-dual interior point method from \cite{AJF}, where in each iteration a linear equality system must be solved. As this cannot be always done exactly, this method accumulates approximations with each iteration, while the proximal point one we propose operates with exact formulae.\\

%does not precisely deliver the optimal solution but an approximation of it.
% Die zugehÃ¶rige Matrix kann auch der Grund dafÃ¼r sein, weshalb bei einigen Beispielen der Algorithmus nicht gegen die "exakte" LÃ¶sung %konvergiert. Hier kÃ¶nnte man noch die %Konditionszahl berechnen und schauen, ob diese Vermutung stimmt.\\

In the following we present the framework of this work and introduce the main tools and results employed later. All the vectors we use are column vectors transposed to row ones by an upper index $^\top$. The space of reals is extended to $\overline{\R}=\mathbb{R}\cup\{\pm\infty\}$, on which, as usual in convex analysis, the following conventions are considered $+\infty + (-\infty) = -\infty + (+\infty) = +\infty$, $0\cdot (+\infty) = +\infty$ and $0\cdot (-\infty) = 0$. A \textit{cone} is a nonempty set $K\subseteq \mathbb{R}^n$ such that $tx\in K$ for all $x\in K$ and all $t\geq 0$. It induces a partial ordering ``$\leqq_K$'' on $\mathbb{R}^n$ by $x\leqq_K y$ if and only if $y-x\in K$, with $x, y \in \mathbb{R}^n$, and an \textit{infinite} element $\infty_K \notin \mathbb{R}^n$ such that $x\leqq_K \infty_K$ and $x\neq \infty_K$ for all $x\in \mathbb{R}^n$ is added to it. Given a set $U\subseteq \mathbb{R}^n$, by $\ri U$ we denote its \textit{relative interior}, that is its interior relative to its affine hull, $P_U$ is the \textit{projection} onto the set $U$, which is, when $U$ is closed and convex, uniquely determined, while its \emph{indicator function}
$\delta_{U}:\mathbb{R}^n\rightarrow \overline{\R}$ is defined by
\begin{displaymath} \delta_{U}(x) = \left\{ \begin{array}{ll} 0, \qquad
\textnormal{if } x \in U,\\
+ \infty,  \textnormal{ otherwise}, \end{array} \right.   \end{displaymath}
and its \emph{support function} $\sigma_{U}:\mathbb{R}^n\rightarrow \overline{\R}$ by $\sigma_{U}(x^*)=\sup_{x\in U}x^{*\top} x$.
The \emph{normal cone} of $U$ at $x\in \mathbb{R}^n$ is $N_U(x):=\{x^*\in \mathbb{R}^n: x^{*\top} (y-x)\leq 0$ $\forall~ y\in U\}$ if $x\in U$ and $N_U(x)=\emptyset$ otherwise.

A function $f:\mathbb{R}^n\rightarrow \overline{\R}$ is called \emph{proper} if its \emph{effective domain} $\dom f:=\{x\in \mathbb{R}^n:~f(x)<+\infty\}$ is not empty and $f(x)>-\infty$ for all $x \in \mathbb{R}^n$. Further, the set $\epi f = \{(x, r)\in \mathbb{R}^n\times \R: f(x)\leq r\}$ is called the \emph{epigraph} of $f$. Remember that a function $f:\mathbb{R}^n\rightarrow \overline{\R}$ is \emph{convex} if $f(\lambda x+(1-\lambda) y)\leq \lambda
f(x)+(1-t)f(y)$ for all $x,y\in \mathbb{R}^n$ and all $t\in[0,1]$ and this happens if and only if $\epi f$ is a convex set.
The \emph{conjugate function} of $f:\mathbb{R}^n\rightarrow \overline{\R}$ is
$f^{*}:\mathbb{R}^n \rightarrow \overline{\R}$, $f^{*}(x^{*})=\sup_{x \in \mathbb{R}^n}\{x^{*\top}x-f(x)\}$, and one has
 the \emph{Young-Fenchel inequality} $f(x)+f^*(x^*)\geq x^{*\top}x$ for all $x, x^*\in \mathbb{R}^n$.
For $x\in \mathbb{R}^n$ such that $f(x)\in \R$, the set $\partial f(x):=\{x^*\in \mathbb{R}^n:f(y)-f(x)\geq \langle x^*,y-x\rangle ~\forall y\in \mathbb{R}^n\}$ denotes the (convex) \emph{subdifferential} of $f$ at $x$, where the elements of this set are called \emph{subgradients}. Moreover, if $\partial f(x)\neq \emptyset$, we say that $f$ is \emph{subdifferentiable at $x$}, while if $f(x)\notin \R$ we take, by convention, $\partial f(x)=\emptyset$. The subgradients of $f$ can be characterized by means of its conjugate, more precisely $x^*\in \partial f(x)$ if and only if $f(x)+f^*(x^*)=x^{*\top} x$, i.e. the Young-Fenchel inequality is fulfilled as an equality for $x$ and $x^*$. Note that for $U\subseteq \mathbb{R}^n$ it holds $\sigma_U = \delta_U^*$ and $N_U=\partial \delta_U$. We take by convention (as usual in entropy optimization) $0\ln 0=0$.
For an optimization problem $(P)$ we denote by $v(P)$ its optimal objective value.

In order to introduce risk measures, consider a probability space $(\Omega, \mathcal{F}, \mathbb{P})$, where $\Omega$ is a set of simple events, $\mathcal{F}$ a $\sigma$-algebra of subsets of $\Omega$ and $\mathbb{P}$ a probability measure on $\mathcal{F}$. Let $L^0$ denote the set of all Borel measurable functions (random variables) $X:\Omega\to \mathbb{R}$, and let the model space $\mathbf{X}$, which is a subspace of $L^0$ including the real numbers (seen as constant functions). For $p\geq 1$, the set of all random variables $X$ for which $\mathbb{E}(|X|^p) = \int |X|^p d\mathbb{P} <+\infty$ is denoted by $L^p$, while the set of all random variables whose moment-generating function $M_X(\theta)=\mathbb{E}(e^{\theta X})$ is finite for all $\theta\in \mathbb{R}$ by $L^M$. Note that for all $p\geq 1$ it holds $L^M\subseteq L^p\subseteq L^0$.
A \textit{risk measure} is a function $\rho:\mathbf{X}\to \overline{\mathbb{R}}$ (which is expected to have some appropriate properties). When $\rho$ is translation invariant, subadditive, monotone and positive homogenous, it is said to be a \textit{coherent risk measure}. The \textit{entropic value-at-risk} ($\EVaR$) (not to be confused with the \textit{entropic risk measure} considered, for instance, in \cite{FOK}) is a coherent risk measure introduced by Ahmadi-Javid in \cite{AJM, AAJ} (see also \cite{AJF, AJP}) and is defined at risk level $\alpha\in (0, 1]$ as
$$\EVaR : L^M\to \overline {\mathbb{R}},\ \EVaR (X)=\inf_{\theta >0} \left\{\theta^{-1} \ln \frac{M_X(\theta)}{\alpha}\right\},\ X\in L^M.$$
By $1_A:\Omega\to \overline{\R}$ we denote the standard \textit{characteristic function} of a set $A\subseteq \Omega$ (not to be confused with the indicator function of $A$) defined by $1_A(\omega) = 1$ if $\omega\in A$ and $1_A (\omega)=0$ otherwise. The \textit{essential supremum} of $X\in L^p$ (where $p\geq 1$) is $\ess\sup(X) = \inf \{t\in \mathbb{R}: \mathbb{P}(\omega: X(\omega)<t)=0\}$.

For the investigations involving entropy functions we also need to introduce the \textit{Lambert $W$ function}, that is the real analytic inverse of the function $x\in \R\mapsto xe^x\in \R$ (see also \cite{LAM} as well as \cite{BAL, BLI} for more on its recent usage in optimization). Even if it does not have an explicit analytical formula and it is a two-valued function, this function was successfully implemented in Mathematics software such as \textsc{matlab} or \textsc{maple} and can be used for computational purposes. It \textit{principal branch}, i.e. the branch taking nowhere values below $-1$ (as seen in \cite[Figure 1]{BLI}), is denoted by $W$ (in the literature also $W_0$), while $W_{-1}$ is the \textit{other real branch}. Among its properties, we mention $W(x)e^{W(x)}=x$ and $x=W(e^x) + \ln W(e^x)$ for all $x\in \mathbb{R}$, while for $x\in [1/e, 0)$ it holds $W_{-1}(x)\leq -1\leq W(x)$.

\section{Entropy Constrained Optimization Problems}\label{se2}

To pave the path to providing a splitting proximal point method for solving the portfolio optimization problem involving the entropic value-at-risk from \cite{AJF}, we focus in this section on the following convex entropy constrained optimization problem
\begin{eqnarray*}
(P)~~\inf_{y\in \mathbb{R}^d,~x\in \mathbb{R}_+^m,~A_px+By\leqq h,\atop \sum\limits_{j=1}^mp_jx_j=1,~\sum\limits_{j=1}^mp_jx_j\ln x_j\leq -H}\left\{c_p^{\top} x+\sigma_D(y)\right\},
\end{eqnarray*}
where $x=(x_1,\hdots,x_m)^\top\in \mathbb{R}^m,~y=(y_1,\hdots,y_d)^\top\in \mathbb{R}^d$ and $0<p_j\leq 1$, $j=1,\hdots,m$, $c_p=(c_1p_1,\hdots p_mc_m)^{\top}\in \mathbb{R}^m,~c=(c_1,\hdots,c_m)^{\top}\in \mathbb{R}^m$, $H\in \mathbb{R}$, $A_p=(a_{ij}p_j)_{i=1, \hdots, n,~j=1, \hdots, m}\in \mathbb{R}^{n\times m}$, where $a_{ij}\in \R$, for all $i=1, \hdots, n$ and $j=1, \hdots, m$, $B=(b_{ij})_{i=1, \hdots, n,~j=1, \hdots, d}\in \mathbb{R}^{n\times d}$, $h=(h_1,\hdots,h_n)^\top\in \mathbb{R}^n$ are given. Further, let $D\subseteq \mathbb{R}^d$ be a nonempty closed convex set. Inspired by \cite[Remark 2.1]{BGW}, we discuss below how to choose $H$ to prevent trivial cases. In Section \ref{se4} we propose a splitting proximal point method for solving $(P)$ by adapting the Douglas-Rachford parallel splitting algorithm provided in \cite{BCO}.

\begin{remark}\normalfont\label{re1}
The motivation behind studying this optimization problem is somehow indirect, although the problem is interesting on its own as well. The starting point of our work was to solve portfolio optimization problems where the risk is modelled by the entropic value-at-risk by means of proximal point methods. Noticing a certain resemblance of the portfolio optimization problem considered in \cite{AJF} to the \textit{Fenchel-Lagrange dual problem} (see \cite{DVO}) obtained in \cite{BGW}, we consider a generalization of the primal entropy constrained linear optimization problem treated in the latter, that is $(P)$. Later we see that its \textit{Lagrange dual problem} turns out in a special case to be the mentioned portfolio optimization problem \cite[(4.7)]{AJF}.
\end{remark}

\begin{remark}\normalfont\label{re2}
By using Lagrange duality, it can easily be verified that the optimal solution of the optimization problem that consists in maximizing the \textit{weighted Shannon entropy} (see, for instance, \cite{GUI}) subject to linear constraints
\begin{eqnarray*}
~~\max_{x_j\geq 0,~j=1,\hdots,m,\atop \sum\limits_{j=1}^m p_jx_j=1} \left\{-\sum_{j=1}^mp_jx_j\ln x_j\right\},
\end{eqnarray*}
is $\bar x_i=1/(\sum_{j=1}^mp_j),~i=1,\hdots,m$, and its optimal objective value is equal to $\ln (\sum_{j=1}^mp_j)$.
Hence, if $H>\ln (\sum_{j=1}^mp_j)$, then the feasible set of $(P)$ is empty. Moreover, if $H=\ln (\sum_{j=1}^mp_j)$, then the $x$-components of any optimal solution $(\bar x, \bar y)$ to $(P)$ are $\bar x_j=1/(\sum_{j=1}^mp_j)$, $j=1,\hdots,m$, and $(P)$ reduces to the unconstrained convex optimization problem
\begin{eqnarray}\label{pr2}
~~c_p^{\top} \bar x+\inf_{y\in \mathbb{R}^d,~By\leqq h-A_p\overline{x}} \sigma_D(y),
\end{eqnarray}
iteratively solvable by existing proximal splitting algorithms.%555
% while the Fenchel dual problem to the optimization problem in \eqref{pr2} is
%\begin{eqnarray*}%\label{dual-pr2}
%~~\sup\limits_{\substack{\lambda=(\lambda_1,\hdots,\lambda_d)^\top\in \mathbb{R}^d\cap D}}
%\left\{-\sigma_C(-\lambda)\right\},
%\end{eqnarray*}
%where $C=\{y\in \mathbb{R}^d:By\leqq h-A\overline{x}\}$.

Finally, to determine the minimal value of $-\sum_{j=1}^mp_jx_j\ln x_j$ subject to $\sum_{j=1}^mp_jx_j=1$, $x_j\geq 0$, $j=1,\hdots, m$, we set $z_j=p_jx_j,~j=1,\hdots,m$, and get the following minimization problem consisting in minimizing the \textit{Kullback-Leibler entropy} subject to the usual linear constraints
\begin{equation}\label{eq3}
~~\inf_{z_j\geq 0,~j=1,\hdots,m,\atop \sum\limits_{j=1}^mz_j=1}\left\{-\sum\limits_{j=1}^mz_j\ln\left(\frac{z_j}{p_j}\right)\right\}=\min_{z_j\geq 0,~j=1,\hdots,m,\atop \sum\limits_{j=1}^mz_j=1}\left\{-\sum\limits_{j=1}^m(z_j\ln z_j-z_j\ln p_j)\right\}.
\end{equation}
For attaining the minimum in \eqref{eq3}, the sum $-\sum_{j=1}^mz_j\ln z_j\geq 0$ must be equal to zero and, as $z_j\ln p_j\leq 0$ for all $j=1,\hdots,m$, an optimal solution to \eqref{eq3} is $\bar z_l=1$, $\bar z_j=0$ for all $j=1,\hdots,m,~j\neq l$, where $l = \min\{i\in \{1,\hdots,m\}: p_i=\min\{p_1,\hdots,p_m\}\}$. Hence, we have that $\bar x_l=1/p_l$, $\bar x_j=0$ for all $j=1,\hdots,m$, $j\neq l$, solves $(P)$, whose minimal value is $\min\{\ln p_1,\hdots,\ln p_m\}=\ln p_l$. In this case, the constraint involving entropy of $(P)$ is automatically fulfilled, and $(P)$ collapses into a convex optimization problem with affine constraints similar to \eqref{pr2}, which can be solved via existing splitting proximal point methods.
\end{remark}

To exclude the trivial cases identified in Remark \ref{re2}, we assume for the rest of the paper that
$$%\begin{eqnarray*}
\min\{\ln p_1,\hdots,\ln p_m\}<H<\ln \sum_{j=1}^mp_j.
$$%\end{eqnarray*}

\begin{remark}\normalfont\label{re3}
Note that the objective function of the problem $(P)$ is proper, convex and lower semicontinuous. For any fixed $y\in \mathbb{R}^n$, the set % and the pointwise maximum of finitely many continuous functions is also continuous. Moreover, as the set
$$
\mathcal{A}=\left\{x\in \mathbb{R}_+^m:A_px+By\leqq h,~\sum\limits_{j=1}^mp_jx_j=1,~ \sum\limits_{j=1}^mp_jx_j\ln x_j\leq -H\right\}
$$
is closed, convex and bounded, thus also compact, so one gets by the generalized Weierstrass theorem that the infimum regarding $x\in \mathcal{A}$ is attained in the optimization problem $(P)$ for any $y\in \mathbb{R}^n$. %has at least one optimal solution and we can write $\min$ instead of $\inf$.
\end{remark}

The Lagrange dual problem corresponding to $(P)$ is
\begin{eqnarray}
(D_L)&&\sup_{\gamma\in \mathbb{R},~\beta\geq 0,
\atop \lambda=(\lambda_1,\hdots,\lambda_n)^\top\in \mathbb{R}^n_+}\inf_{y\in \mathbb{R}^d,~x\in \mathbb{R}_+^m}\Bigg\{c_p^{\top} x+\sigma_D(y)+\lambda^{\top}\left(A_px+By-h\right)\notag\\
&&+\gamma\left(\sum_{j=1}^mp_jx_j-1\right)
+\beta\Bigg(\sum_{j=1}^mp_jx_j\ln x_j+H\Bigg)\Bigg\},\nonumber
\end{eqnarray}
that can be simplified to
$$\sup_{\gamma\in \mathbb{R},~\beta\geq 0, \atop \lambda=(\lambda_1,\hdots,\lambda_n)^\top\in \mathbb{R}^n_+}\Bigg\{-\lambda^{\top}h+\inf_{y\in \mathbb{R}^d}\Big\{\lambda^{\top}By+\sigma_D(y)\Big\}$$
$$+\inf_{x\in \mathbb{R}^n_+}\Bigg\{c_p^{\top} x+\lambda^{\top}A_p x+\gamma\sum_{j=1}^mp_jx_j+\beta\sum_{j=1}^m p_jx_j\ln x_j\Bigg\}+\beta H-\gamma\Bigg\}$$
end even to
$$
(D_L)\  \ \sup_{\gamma\in \mathbb{R},~\beta\geq 0,~ -B^{\top}\lambda \in D \atop \lambda=(\lambda_1,\hdots,\lambda_n)^\top\in \mathbb{R}^n_+}\Bigg\{\sum_{j=1}^mp_j\inf_{x_j\in \mathbb{R}_+}\left\{\left(\sum_{i=1}^n \lambda_ia_{ij}+c_j+\gamma\right)x_j+\beta x_j\ln x_j\right\}+\beta H-\lambda^{\top}h-\gamma\Bigg\},$$
where the last equality holds since $\inf_{y\in \mathbb{R}^n}\{\lambda^{\top} By+\sigma_D(y)\}=-\sigma_D^*(-B^{\top}\lambda)=-\delta_D(-B^{\top}\lambda)$ for $\lambda\in \mathbb{R}^n_+$, as $D$ is closed and convex.

Let us analyze now the role of the variable $\beta$ in the dual problem $(D_L)$. If $\beta>0$, it becomes
\begin{eqnarray*}
(D^1_L)&\sup\limits_{\substack{\gamma\in \mathbb{R},~\beta>0,\\ \lambda=(\lambda_1,\hdots,\lambda_n)^\top\in \mathbb{R}_+^n,\\ B^{\top}\lambda \in D}}\Bigg\{\sum\limits_{j=1}^mp_j\inf_{x_j\in \mathbb{R}_+}\left\{\left(\sum\limits_{i=1}^n \lambda_ia_{ij}+c_j+\gamma\right)x_j+\beta x_j\ln x_j\right\}+\beta H-\lambda^{\top}h-\gamma\Bigg\}\\
=&\sup\limits_{\gamma\in \mathbb{R},~\beta>0, B^{\top}\lambda \in D\atop \lambda =(\lambda_1,\hdots,\lambda_n)^\top\in \mathbb{R}_+^n}\left\{-\sum\limits_{j=1}^m\beta p_je^{-\frac{\sum\limits_{i=1}^n\lambda_ia_{ij}+c_j+\gamma+\beta }{\beta }}+\beta H-\lambda^{\top}h-\gamma\right\},
\end{eqnarray*}
where the last equality follows from the fact that the function $\varphi(x):\mathbb{R}_+\rightarrow \mathbb{R}$, $\varphi(x)=v x\ln x+wx$ with $v>0,~w\in \mathbb{R}$, attains its minimum at $e^{-\frac{v+w}{v}}$. The corresponding minimal function value is then given by $\varphi(e^{-\frac{v+w}{v}})=-ve^{-\frac{v+w}{v}}$ (see \cite{BGW}). In summary, one gets
\begin{eqnarray*}
(D^1_L)&& \sup_{\gamma\in \mathbb{R},~\beta>0,\atop \lambda\in \mathbb{R}_+^n,~-B^{\top}\lambda \in D}\left\{-\sum_{j=1}^m\beta p_je^{-\frac{\lambda^\top A_j +c_j+\gamma+\beta }{\beta}}+\beta H-\lambda^{\top}h-\gamma\right\},
\end{eqnarray*}
where $A_j=(a_{1j},\hdots,a_{nj})^\top\in \mathbb{R}^n$ is the $j$-th column of the matrix $A=(a_{ir})_{i=1, \hdots, n,~r=1, \hdots, m}\in \mathbb{R}^{n\times m}$, $j=1,\hdots,m$. Further, take attention to the following function $q:\mathbb{R}\rightarrow \mathbb{R},~q(\gamma)=-\beta we^{-\frac{\gamma}{\beta}}-\gamma+k$ with $w,~\beta>0$ and $k\in \mathbb{R}$, which attains its maximum at $\gamma=\beta\ln w$, that is  $q(\beta\ln w)=k-\beta\ln(ew)$. Setting
\begin{eqnarray*}
w=\sum_{j=1}^mp_je^{-\frac{\lambda^\top A_j+c_j+\beta}{\beta}}\text{ and }k=\beta H - \lambda^{\top}h,
\end{eqnarray*}
turns $(D^1_L)$ into
\begin{eqnarray*}
(D^1_L)&&\sup_{\beta>0,~\lambda\in \mathbb{R}_+^n,\atop -B^{\top}\lambda \in D}\left\{-\beta\ln\left(\sum_{j=1}^m p_je^{-\frac{\lambda^\top A_j+c_j+\beta }{\beta }+1}\right)+\beta H-\lambda^{\top}h\right\}\\
&=& \sup_{\beta>0,~\lambda\in \mathbb{R}_+^n,\atop -B^{\top}\lambda \in D}\left\{-\beta\ln\left(\sum_{j=1}^m p_je^{-\frac{\lambda^\top A_j +c_j }{\beta }}\right)+\beta H-\lambda^{\top}h\right\}.
\end{eqnarray*}

On the other hand, when $\beta=0$ in $(D_L)$, then $\min_{x_j\in \mathbb{R}_+}\left\{(\sum_{i=1}^n \lambda_ia_{ij}+c_j+\gamma)x_j\right\}=0$ if and only if $\sum_{i=1}^n\lambda_ia_{ij}+c_j+\gamma\geq 0$, otherwise $\min_{x_j\in \mathbb{R}_+}\left\{(\sum_{i=1}^n \lambda_ia_{ij}+c_j+\gamma)x_j\right\}=-\infty$, $j=1,\hdots,m$. The dual problem $(D_L)$ collapses in this situation to
\begin{eqnarray*}
(D^2_L)&&\sup_{\gamma\in \mathbb{R},~\lambda\in \mathbb{R}_+^n,~-B^{\top}\lambda \in D,\atop \sum\limits_{i=1}^n\lambda_ia_{ij}+c_j+\gamma\geq 0,~j=1,\hdots,m}\{-\lambda^{\top}h-\gamma\}=\sup_{\lambda\in \mathbb{R}_+^n,~-B^{\top}\lambda \in D}\left\{\min_{1\leq j\leq m}\left\{\lambda^\top A_j+c_j\right\}-\lambda^{\top}h\right\}.
\end{eqnarray*}

Putting the two cases treated above together, one obtains $v(D_L)=\max \{v(D_L^1), v(D_L^2)\}$. From \cite[Lemma 4.1]{FRT} we derive that
\begin{eqnarray*}
\lim_{\beta \downarrow 0}\beta\left(H-\ln\left(\sum_{j=1}^mp_je^{-\frac{\lambda^\top A_j+c_j}{\beta}}\right)\right)=\min_{1\leq j\leq m}\lambda^\top A_j+c_j,
\end{eqnarray*}
which shows that $v(D_L^1)\geq v(D_L^2)$ and reduces the dual problem $(D_L)$ to
\begin{eqnarray*}
(D_L)&&\sup_{\beta>0,~\lambda\in \mathbb{R}_+^n,\atop -B^{\top}\lambda \in D}\left\{-\beta\ln\left(\sum_{j=1}^m p_je^{-\frac{\lambda^\top A_j+c_j }{\beta }}\right)+\beta H-\lambda^{\top}h\right\}.
\end{eqnarray*}

\begin{remark}\normalfont\label{re5}\normalfont
One can show that $(D_L)$ coincides with the Fenchel-Lagrange dual problem to $(P)$ (see \cite{DVO}).
\end{remark}

By construction, one always has \textit{weak duality}, i.e. $v(D_L)\leq v(P)$. In order to achieve \textit{strong duality} (that is the situation when the optimal objective values of the primal and dual problem coincide and the latter has an optimal solution), we employ the following refinement of the classical \textit{Slater constraint qualification} (cf. \cite{DVO})

\begin{center}
\begin{tabular}{r|l}
$(RC)$ \ & \ $\exists (x,y)\in \ri\left(\mathbb{R}^m_+\times \dom \sigma_D\right) \text{ such that } A_px+By\leqq h,$\\
 \ & \ $\sum\limits_{j=1}^mp_jx_j=1 \text{ and } \sum\limits_{j=1}^mp_jx_j\ln x_j<-H$.
\end{tabular}
\end{center}

The strong duality statement for the primal-dual pair of optimization problems $(P)-(D_L)$ follows directly from \cite[Theorem 3.2.9]{DVO}, and it is followed by the corresponding necessary and sufficient optimality conditions.

\begin{theorem}[strong duality] \label{theorem:1} If the constraint qualification $(RC)$ is fulfilled, then strong duality between $(P)$ and $(D_L)$  holds, i.e.
$v(P) = v(D_L)$ and the dual problem $(D_L)$ has an optimal solution $(\overline{\beta},\overline{\lambda}_1,\hdots,\overline{\lambda}_n)\in(0,+\infty)\times \mathbb{R}^n_+$.\end{theorem}

\begin{theorem}[optimality conditions]\label{theoremChapter1:3} $(a)$ Suppose that the constraint qualification $(RC)$ is fulfilled and let $(\overline{x},\overline{y})\in \mathbb{R}^m_+\times \mathbb{R}^n$ be an optimal solution to the problem $(P)$. Then there
exists $(\overline{\beta},\overline{\lambda}_1,\hdots,\overline{\lambda}_n)\in(0,+\infty)\times \mathbb{R}^n_+$, an optimal solution to $(D_L)$, such that
\renewcommand{\labelenumi}{(\roman{enumi})}\begin{enumerate}
\item $c_p^{\top}\overline{x}+\sigma_D(\overline{y})=-\overline{\beta}\ln\left(\sum\limits_{j=1}^m p_je^{-\frac{\overline{\lambda}^\top A_j+c_j }{\overline{\beta} }}\right)+\overline{\beta} H-\overline{\lambda}^{\top}h
=\min\limits_{y\in \mathbb{R}^d,~ x\in \mathbb{R}_+^m}\Bigg\{c_p^{\top} x+\sigma_D(y)+\overline{\lambda}^{\top}\left(A_px+By-h\right)+\overline{\gamma}\left(\sum\limits_{j=1}^mp_jx_j-1\right)
+\overline{\beta}\Bigg(\sum\limits_{j=1}^mp_jx_j\ln x_j+H\Bigg)\Bigg\}$,
\item $\sigma_D(\overline{y})=-\overline{\lambda}^\top B\overline{y}$,
\item $\overline{\lambda}^{\top}\left(A_p\overline{x}+B\overline{y}-h\right)=0$,
\item $\sum\limits_{j=1}^mp_j\overline{x}_j\ln \overline{x}_j=-H$,
\item $\overline{\beta}\sum\limits_{j=1}^mp_j\overline{x}_j\ln \overline{x}_j+\overline{\beta}\ln\left(\sum\limits_{j=1}^mp_je^{-\frac{\overline{\lambda}^\top A_j+c_j
}{\overline{\beta}}}\right)=-c_p^{\top}\overline{x}-\overline{\lambda}^{\top}A_p\overline{x}$,
\item $\overline{\lambda}_i\geq 0,~i=1,\hdots,n,~\overline{\beta}>0,~-B^{\top}\overline{\lambda}\in D\textnormal{ and } \sum\limits_{j=1}^mp_j\overline{x}_j=1.$
\end{enumerate}
$(b)$ If there exists $(\overline{x},\overline{y})\in \mathbb{R}^m_+\times \mathbb{R}^n$ such that for some $(\overline{\beta},\overline{\lambda}_1,\hdots,\overline{\lambda}_n)\in(0,+\infty)\times \mathbb{R}^n$ the
conditions $(i)$-$(vi)$ are fulfilled, then $(\overline{x},\overline{y})$ is an optimal solution to $(P)$, $(\overline{\beta},\overline{\lambda}_1,\hdots,\overline{\lambda}_n)$
is an optimal solution to $(D_L)$ and $v(P) = v(D_L)$.
\end{theorem}

\begin{proof} $(a)$ As $(\overline{x},\overline{y})\in \mathbb{R}^m_+\times \mathbb{R}^n$ is an optimal solution to $(P)$, and by Theorem \ref{theorem:1} there exists an optimal solution to $(D_L)$ denoted $(\overline{\beta},\overline{\lambda}_1,\hdots,\overline{\lambda}_n)\in(0,+\infty)\times \mathbb{R}^n_+$ such that
\begin{eqnarray*}
&&c_p^{\top}\overline{x}+\sigma_D(\overline{y})=-\overline{\beta}\ln\left(\sum_{j=1}^m p_je^{-\frac{\overline{\lambda}^\top A_j+c_j }{\overline{\beta} }}\right)+\overline{\beta} H-\overline{\lambda}^{\top}h\\
&\Leftrightarrow& c_p^{\top}\overline{x}+\sigma_D(\overline{y})+\overline{\beta}\ln\left(\sum_{j=1}^m p_je^{-\frac{\overline{\lambda}^\top A_j+c_j} {\overline{\beta} }}\right)-\overline{\beta} H+\overline{\lambda}^{\top}h\\
&&+\overline{\lambda}^{\top}\left(A_p\overline{x}+B\overline{y}\right)-\overline{\lambda}^{\top}\left(A_p\overline{x}+B\overline{y}\right)+\overline{\beta}\sum\limits_{j=1}^mp_j\overline{x}_j\ln \overline{x}_j-\overline{\beta}\sum\limits_{j=1}^mp_j\overline{x}_j\ln \overline{x}_j=0%\notag\\
\end{eqnarray*}
\begin{eqnarray}
&\Leftrightarrow& \left[\sigma_D(\overline{y})+\overline{\lambda}^{\top}B\overline{y}\right]+\left[-\overline{\lambda}^{\top}\left(A_p\overline{x}+B\overline{y}-h\right)\right]+
\left[-\overline{\beta}\left(H+\sum\limits_{j=1}^mp_j\overline{x}_j\ln \overline{x}_j\right)\right]\notag\\
&&+\left[\overline{\beta}\sum\limits_{j=1}^mp_j\overline{x}_j\ln \overline{x}_j+\overline{\beta}\ln\left(\sum\limits_{j=1}^mp_je^{-\frac{\overline{\lambda}^\top A_j+c_j
}{\overline{\beta}}}\right)+c_p^{\top}\overline{x}+\overline{\lambda}^{\top}A_p\overline{x}\right]=0.\label{equal}
\end{eqnarray}

From the Young-Fenchel inequality we have $\sigma_D(\overline{y})+\delta_D(-B^{\top}\overline{\lambda})+\overline{\lambda}^\top B\overline{y}\geq 0$, which shows that the first bracket in \eqref{equal} is nonnegative. Due to the feasibility of $\overline{\lambda}$, $\overline{\beta}$ and $\overline{x}$ we have that $\overline{\lambda}\geq0,~\overline{\beta}>0$ and $A_p\overline{x}+B\overline{y}\leqq h$ as well as $H+\sum_{j=1}^mp_j\overline{x}_j\ln \overline{x}_j\leq 0$, which implies the nonnegativity of the second and the third bracket. The nonnegativity of the last bracket can be shown by using the functions $\varphi$ and $q$ defined before when deriving the final formulation of  $(D_L)$. Analogously to the considerations on the function $\varphi$ concerning its minimal value, one has for all $s\in \mathbb{R}$ the following inequality
$$
\sum\limits_{j=1}^mp_j\left(\overline{\beta} \overline{x}_j\ln \overline{x}_j+\overline{x}_j\left(\sum\limits_{i=1}^n\overline{\lambda}_ia_{ij}+c_j+s\right)\right)-s\geq -s-\overline{\beta}\sum\limits_{j=1}^mp_je^{-\frac{\overline{\beta}+\sum\limits_{i=1}^n\overline{\lambda}_ia_{ij}+c_j-s}{\overline{\beta}}}.
$$
As $\sum_{j=1}^mp_j\overline{x}_j=1$, the variable $s$ vanishes in the left side of the previous inequality, hence
\begin{eqnarray*}
\sum\limits_{j=1}^mp_j\left(\overline{\beta} \overline{x}_j\ln \overline{x}_j+\overline{x}_j\left(\sum\limits_{i=1}^n\overline{\lambda}_ia_{ij}+c_j\right)\right)\geq \sup_{s\in \mathbb{R}}\left\{-s-\overline{\beta}\sum\limits_{j=1}^mp_je^{-\frac{\overline{\beta}+\sum\limits_{i=1}^n\overline{\lambda}_ia_{ij}+c_j+s}{\overline{\beta}}}\right\}.
\end{eqnarray*}
Finally, taking into account the results regarding the maximal value of the function $q$ investigated above, leads to
$$
\sum\limits_{j=1}^mp_j\left(\overline{\beta} \overline{x}_j\ln \overline{x}_j+\overline{x}_j\left(\sum\limits_{i=1}^n\overline{\lambda}_ia_{ij}+c_j\right)\right)\geq -\overline{\beta}\ln\left(\sum\limits_{j=1}^mp_je^{-\frac{\sum\limits_{i=1}^n\overline{\lambda}_ia_{ij}+c_j}{\overline{\beta}}}\right)$$
that is equivalent to
$$\overline{\beta}\sum\limits_{j=1}^mp_j\overline{x}_j\ln \overline{x}_j+c_p^{\top}\overline{x}+\overline{\lambda}^{\top}A_p\overline{x}\geq-\overline{\beta}\ln\left(\sum\limits_{j=1}^m
p_je^{-\frac{\overline{\lambda}^\top A_j+c_j}{\overline{\beta}}}\right),
$$
which proves that the last bracket is also nonnegative. In summary, all brackets in \eqref{equal} are nonnegative and must therefore be equal to zero, i.e. the optimality conditions $(i)-(vi)$ are fulfilled. Clearly, optimality condition $(vi)$ is valid by feasibility.

$(b)$ All calculations done in $(a)$ also hold in the reverse direction.\end{proof}

\section{Special case}\label{se3}

The following special case brings us closer to the portfolio optimization problem investigated in \cite[Section 4.2]{AJF}. Setting $d=n+1$, $c_i=0$, $i=1, \hdots, m$, $D=\widetilde{D}\times\{1\}$ with $\widetilde{D}\subseteq \mathbb{R}^n$, $h_j=0$, $j=1, \hdots, n$, and $B=-\widetilde{B}\in \mathbb{R}^{n\times n+1}$ with
$$\widetilde{B}=\begin{pmatrix}
1& \hdots &0 &1  \\
\vdots &\ddots    &\vdots& \vdots\\
0& \hdots &1&1
\end{pmatrix}$$
transforms $(P)$ into the following optimization problem
\begin{eqnarray*}
(P')&&\inf_{y\in \mathbb{R}^{n+1},~x\in \mathbb{R}_+^m,~A_px-\widetilde{B}y \leqq 0_{\mathbb{R}^n},\atop \sum\limits_{j=1}^mp_jx_j=1,~ \sum\limits_{j=1}^mp_jx_j\ln x_j\leq -H}\left\{\sigma_{\widetilde{D}\times\{1\}}(y)\right\}\\
&=&\inf_{y_{n+1}\in \mathbb{R},~\widetilde{y}=(y_1,\hdots,y_n)^{\top}\in \mathbb{R}^n,~x\in \mathbb{R}_+^m, \ \sum\limits_{j=1}^mp_jx_j=1,\atop ~\sum\limits_{j=1}^mp_ja_{ij}x_j-y_i\leq y_{n+1},~i=1,\hdots,n,~ \sum\limits_{j=1}^mp_jx_j\ln x_j\leq -H}\left\{y_{n+1}+\sigma_{\widetilde{D}}(\widetilde{y})\right\}.
\end{eqnarray*}
By setting $t=y_{n+1}$ and $y=\widetilde{y}$ one gets
\begin{eqnarray*}
(P')&&\inf_{t\in \mathbb{R},~ y\in \mathbb{R}^n,~x\in \mathbb{R}_+^m,~\sum\limits_{j=1}^mp_ja_{ij}x_j-y_i\leq t,~i=1,\hdots,n,\atop\sum\limits_{j=1}^mp_jx_j=1,~ \sum\limits_{j=1}^mp_jx_j\ln x_j\leq -H}\left\{t+\sigma_{\widetilde{D}}(y)\right\}\\
&=&\inf_{y\in \mathbb{R}^n,~x\in \mathbb{R}_+^m,~\sum\limits_{j=1}^mp_jx_j=1,\atop \sum\limits_{j=1}^mp_jx_j\ln x_j\leq -H}\left\{\max_{1\leq i \leq n}\left\{\sum_{j=1}^mp_ja_{ij}x_j-y_i\right\}+\sigma_{\widetilde{D}}(y)\right\},
\end{eqnarray*}
which demonstrates that the optimization problem $(P)$ also covers a class of minimax problems. Moreover, its Lagrange dual problem turns out to be a portfolio optimization problem where the risk is modelled by the entropic value-at-risk under the sample-based setting as investigated in \cite[Section 4.2]{AJF}.\\

By the definitions made in this section, noting that
\begin{eqnarray}
-B^{\top}\lambda\in D\Leftrightarrow \widetilde{B}^{\top}\lambda\in \widetilde{D}\times\{1\} \Leftrightarrow \lambda\in \widetilde{D},~\sum_{i=1}^n\lambda_i=1,\label{setD}
\end{eqnarray}
the dual problem $(D_L)$ reduces to
\begin{eqnarray*}
(D'_L)&&\sup_{\lambda = (\lambda_1, \hdots, \lambda_n)^\top\in \widetilde{D}\cap \R^n_+, \atop \sum\limits_{i=1}^n\lambda_i=1,~  \beta>0}\left\{-\beta\ln\left(\sum_{j=1}^m p_je^{-\frac{\lambda^\top A_j }{\beta }}\right)+\beta H\right\},
\end{eqnarray*}
whose minimization counterpart is the portfolio optimization problem solved in \cite{AJF} via a primal dual interior point algorithm.

\begin{remark}\normalfont\label{re4}\normalfont
When $\widetilde{D}$ does not intersect the unit simplex, the dual problem $(D'_L)$ is infeasible, in which case strong duality would yield the unboundedness from below of $(P')$.
\end{remark}

Since the situation described in Remark \ref{re4} is not relevant in portfolio optimization, we assume for the rest of the paper that
$$\mathcal{M}=\widetilde{D} \cap \left\{\lambda=(\lambda_1, \hdots, \lambda_n)^\top\in \R^n_+: \sum\limits_{i=1}^n \lambda_i=1\right\}\neq \emptyset.$$

As $\dom\sigma_D=\dom(\sigma_{\widetilde{D}}+\cdot)=\dom\sigma_{\widetilde{D}}\times \mathbb{R}$, the refined Slater constraint qualification turns into
\begin{center}
\begin{tabular}{r|l}
$(RC')$ \ & \ $\exists (x,y,t)\in \ri\left(\mathbb{R}^m_+\times \dom \sigma_{\widetilde{D}}\right)\times \mathbb{R} \text{ such that } \sum\limits_{j=1}^mp_ja_{ij}x_j-y_i\leq t,~i=1,\hdots,n,$\\
\ & \ $\sum\limits_{j=1}^mp_jx_j=1 \text{ and } \sum\limits_{j=1}^mp_jx_j\ln x_j<-H$.
\end{tabular}
\end{center}

The next two statements are direct consequences of Theorem \ref{theorem:1} and Theorem \ref{theoremChapter1:3}.

\begin{theorem}[strong duality] \label{theorem:19} If the constraint qualification $(RC')$ is fulfilled, then strong duality between $(P')$ and $(D'_L)$  holds, i.e.
$v(P') = v(D'_L)$ and the dual problem $(D'_L)$ has an optimal solution $(\overline{\beta},\overline{\lambda}_1,\hdots,\overline{\lambda}_n)\in(0,+\infty)\times \mathcal{M}$.\end{theorem}

\begin{theorem}[optimality conditions]\label{theoremChapter1:3.2} $(a)$ Suppose that the constraint qualification $(RC')$ is fulfilled and let $(\overline{x},\overline{y},\overline{t})\in \mathbb{R}^m_+\times \mathbb{R}^n\times \mathbb{R}$ be an optimal solution to the problem $(P')$. Then there
exists $(\overline{\beta},\overline{\lambda}_1,\hdots,\overline{\lambda}_n)\in(0,+\infty)\times \mathcal{M}$, an optimal solution to $(D'_L)$, such that
\renewcommand{\labelenumi}{(\roman{enumi})}\begin{enumerate}
\item $\overline{t}+\sigma_{\widetilde{D}}(\overline{y})=-\overline{\beta}\ln\left(\sum_{j=1}^m p_je^{-\frac{\overline{\lambda}^\top A_j }{\overline{\beta} }}\right)+\overline{\beta} H=\min\limits_{t\in \mathbb{R},~y\in \mathbb{R}^n,~x\in \mathbb{R}_+^m}\Bigg\{t+\sigma_{\widetilde{D}}(y)\\
    +\sum\limits_{i=1}^n\overline{\lambda}_i\left(\sum\limits_{j=1}^mp_ja_{ij}x_j
    -y_i-t\right)+\overline{\gamma}\left(\sum\limits_{j=1}^mp_jx_j-1\right)
+\overline{\beta}\Bigg(\sum\limits_{j=1}^mp_jx_j\ln x_j+H\Bigg)\Bigg\},$
\item $\sigma_{\widetilde{D}}(\overline{y})=-\overline{\lambda}^\top\overline{y}$,
\item $\overline{t}=\sum\limits_{i=1}^n\overline{\lambda}_i\left(\sum\limits_{j=1}^mp_ja_{ij}\overline{x}_j
    -\overline{y}_i\right)=\max\limits_{1\leq i\leq n}\left\{\sum\limits_{j=1}^mp_ja_{ij}\overline{x}_j-\overline{y}_i\right\}$,
\item $\sum\limits_{j=1}^mp_j\overline{x}_j\ln \overline{x}_j=-H$,
\item $\overline{\beta}\sum\limits_{j=1}^mp_j\overline{x}_j\ln \overline{x}_j+\overline{\beta}\ln\left(\sum\limits_{j=1}^mp_je^{-\frac{\overline{\lambda}^\top A_j
}{\overline{\beta}}}\right)=-\sum\limits_{i=1}^n\overline{\lambda}_i\left(\sum\limits_{j=1}^mp_ja_{ij}\overline{x}_j\right)$,
\item $\overline{\lambda}_i\geq 0,~i=1,\hdots,n,~\overline{\lambda}\in \widetilde{D}$, $\sum\limits_{i=1}^n\overline{\lambda}=1$, $\overline{\beta}>0$,
\item $\overline{\lambda}_i\left(\sum\limits_{j=1}^mp_ja_{ij}\overline{x}_j-\overline{y}_i-\overline{t}\right)=0$, $i=1,\hdots,n$.
\end{enumerate}
$(b)$ If there exists $(\overline{x},\overline{y},\overline{t})\in \mathbb{R}^m_+\times \mathbb{R}^n\times \mathbb{R}$ such that for some $(\overline{\beta},\overline{\lambda}_1,\hdots,\overline{\lambda}_n)\in(0,+\infty)\times \mathbb{R}^n$ the
conditions $(i)$-$(vii)$ are fulfilled, then $(\overline{x},\overline{y})$ is an optimal solution to $(P')$, $(\overline{\beta},\overline{\lambda}_1,\hdots,\overline{\lambda}_n)$
is an optimal solution to $(D'_L)$ and $v(P') = v(D'_L)$.
\end{theorem}

\begin{proof}
$(a)$ Recall the observation in \eqref{setD} and that $\sigma_D=\sigma_{\widetilde{D}}+\cdot$ as well as $\overline{\lambda} ^{\top}\widetilde{B}(\overline{y},\overline{t})=\overline{\lambda}^{\top}\overline{y}+\overline{t}$ holds. Then the optimality conditions $(i)$-$(vi)$ of Theorem \ref{theoremChapter1:3.2} are direct consequences of Theorem \ref{theoremChapter1:3}. Further, the second equality of the optimality condition $(iii)$ can be achieved from the first equality of the optimality condition $(iii)$, the feasibility of $(\overline{x},\overline{y},\overline{t})$ and the optimality condition $(v)$, i.e. $\sum_{i=1}^n\overline{\lambda}=1$,
\begin{eqnarray}
&&\sum\limits_{i=1}^n\overline{\lambda}_i\left(\sum\limits_{j=1}^mp_ja_{ij}\overline{x}_j
    -\overline{y}_i\right)=\overline{t}\geq\max\limits_{1\leq i\leq n}\left\{\sum\limits_{j=1}^mp_ja_{ij}\overline{x}_j-\overline{y}_i\right\}\notag\\
&=&\sum_{i=1}^n\overline{\lambda}_i\max\limits_{1\leq i\leq n}\left\{\sum\limits_{j=1}^mp_ja_{ij}\overline{x}_j-\overline{y}_i\right\}\geq\sum\limits_{i=1}^n\overline{\lambda}_i\left(\sum\limits_{j=1}^mp_ja_{ij}\overline{x}_j
    -\overline{y}_i\right).\label{tsum}
\end{eqnarray}
Moreover, it holds by \eqref{tsum} that
\begin{eqnarray}
\sum\limits_{i=1}^n\overline{\lambda}_i\left(\overline{t}-\left(\sum\limits_{j=1}^mp_ja_{ij}\overline{x}_j-\overline{y}_i\right)\right)=0\label{tsum2}
\end{eqnarray}
and since $\lambda_i\geq 0$ as well as  $\overline{t}-\left(\sum_{j=1}^mp_ja_{ij}\overline{x}_j-\overline{y}_i\right)\geq0$ for all $i=1,\hdots,n$, one gets that all the terms in \eqref{tsum2} are nonnegative and must therefore be equal to zero, which proves the optimality condition $(vii)$.

$(b)$ All calculations done in $(a)$ hold also in the reverse direction.
\end{proof}

\begin{remark}\normalfont
Clearly, if $0_{\mathbb{R}^n}\in \interior \widetilde{D}$, then $\sigma_{\widetilde{D}}$ is continuous and real-valued and by setting
$$x_j=\frac{1}{\sum\limits_{j=1}^mp_j},~j=1,\hdots,m,~y_i=0,~i=1,\hdots,n,\text{ and }t=\frac{1}{\sum\limits_{j=1}^mp_j}\sum\limits_{i=1}^n\sum\limits_{j=1}^mp_ja_{ij},$$
it is obvious that $(RC')$ is fulfilled, therefore strong duality for the primal-dual pair $(P')-(D'_L)$ holds.
\end{remark}

The next two remarks discuss the question how to determine an optimal solution to the dual problem $(D'_L)$ when an optimal solution to the primal problem $(P')$ is known and vice versa.

\begin{remark}\normalfont Notice that \cite[Remark 3.3.8]{DVO} points up that the optimality condition $(i)$ in Theorem \ref{theoremChapter1:3.2} is equivalent, after some simplifications, to
\begin{equation}\label{equation0}
\left\{
\begin{array}{l}
\sum_{i=1}^n\overline{\lambda}_i=1,~\overline{\lambda}\in \partial\sigma_{\widetilde D}(\overline{y}),\\
\sum_{i=1}^n\overline{\lambda}_ip_ja_{ij}+\overline{\gamma} p_j+\overline{\beta}p_j(\ln \overline{x}_j+1)=0,~j=1,\hdots,m.
\end{array}\right.
\end{equation}
Multiplying the last equality with $\overline{x}_j$ and taking the sum over all $j=1,\hdots,m$, leads to
\begin{eqnarray}
\sum_{j=1}^m\sum_{i=1}^n\overline{\lambda}_ip_ja_{ij}\overline{x}_j+\overline{\gamma}\sum_{j=1}^m p_j\overline{x}_j+\overline{\beta}\sum_{j=1}^m p_j\overline{x}_j(\ln \overline{x}_j+1)=0.\label{equation1}
\end{eqnarray}
From the optimality conditions $(ii)-(iii)$ in Theorem \ref{theoremChapter1:3.2} one gets $$v(P)=\sum\limits_{i=1}^n\sum\limits_{j=1}^m\overline{\lambda}_ip_ja_{ij}\overline{x}_j,$$
while optimality condition $(iv)$ in the same statement yields $-H=\sum_{j=1}^mp_j\overline{x}_j\ln \overline{x}_j$. Using that $\sum_{j=1}^mp_j\overline{x}_j=1$, \eqref{equation1} turns into
\begin{eqnarray}
\overline{\gamma}= -v(P)-\overline{\beta}\left(1-H\right).\label{equation2}
\end{eqnarray}
By combining \eqref{equation0} and \eqref{equation2} one has
\begin{eqnarray*}
&&\sum_{i=1}^n\overline{\lambda}_ip_ja_{ij}-p_j\left(v(P)+\overline{\beta}\left(1-H\right)\right)+\overline{\beta}p_j(\ln \overline{x}_j+1)=0\\
&\Leftrightarrow&\sum_{i=1}^n\overline{\lambda}_ia_{ij}+\overline{\beta}(\ln \overline{x}_j+H)=v(P),~j=1,\hdots,m,
\end{eqnarray*}
therefore the optimal solution $(\overline{\beta},\overline{\lambda})$ to the dual problem $(D_L')$ is the solution of the linear equality system
$$\left\{
\begin{array}{l}
\sum\limits_{i=1}^n\lambda_ia_{ij}+\beta(\ln \overline{x}_j+H)=v(P),~j=1,\hdots,m,\\
\sum\limits_{i=1}^n\lambda_i=1,\\
\lambda\in \widetilde D \cap \mathbb{R}^n_+,\\
  \beta > 0.
\end{array}\right.$$
%Hence, if $D\supseteq \mathbb{R}^n_+$ and $(\overline{x},\overline{y})$ is an optimal solution to the primal optimization problem $(P)$, then the optimal solution $(\overline{\beta},\overline{\lambda})$ to the dual problem $(D_L)$ is the unique solution to a linear equality system.
\end{remark}

\begin{remark}\normalfont
If one has an optimal solution $(\overline{\beta},\overline{\lambda})$ to the dual problem $(D'_L)$, then an optimal solution $(\overline{x},\overline{y},\overline{t})$ to $(P')$ is also determinable by using the optimality conditions given in Theorem \ref{theoremChapter1:3.2}. The component $\overline{x}$ can then be obtained by solving a linear equality system of the form
\begin{eqnarray}
&&\left\{
\begin{array}{l}\sum\limits_{i=1}^n\overline{\lambda}_ia_{ij}+\overline{\beta}(\ln x_j+H)=v(D'_L),~j=1,\hdots,m,\\
\sum\limits_{j=1}^m\sum\limits_{i=1}^n\overline{\lambda}_ip_ja_{ij}x_j=v(D'_L),
\end{array}\right.\notag\\
&\Leftrightarrow&\left\{
\begin{array}{l}x_j=e^{\frac{1}{\overline{\beta}}\left(v(D'_L)-\sum\limits_{i=1}^n\overline{\lambda}_ia_{ij}\right)-H},~j=1,\hdots,m,\\
\sum\limits_{j=1}^m\sum\limits_{i=1}^n\overline{\lambda}_ip_ja_{ij}x_j=v(D'_L),
\end{array}\right.\label{x_solution}
\end{eqnarray}
while $\overline{y}$ and $\overline{t}$ can be deduced from $\overline{y}\in \partial \delta_{\widetilde{D}}(\overline{\lambda})$ and  $\overline{t}=\sum_{i=1}^n\overline{\lambda}_i\left(\sum_{j=1}^mp_ja_{ij}\overline{x}_j-\overline{y}_i\right)$. Alternatively, one can find $(\overline{y},\overline{t})$ by solving (for instance, with a proximal splitting algorithm) the optimization problem
\begin{eqnarray*}
\inf_{y\in \mathbb{R}^n}\left\{\max_{1\leq i \leq n}\left\{\sum_{j=1}^mp_ja_{ij}\overline{x}_j-y_i\right\}+\sigma_{\widetilde{D}}(y)\right\}.
\end{eqnarray*}

Note that a solution $\overline{x}_j$, $j=1,\hdots,m$, of the linear equality system in \eqref{x_solution} fulfills of course the optimality conditions of $(P')$, too. Let $\overline{x}_j$, $j=1,\hdots,m$, be an solution of the linear equality system in \eqref{x_solution}, then by multiplying in both sides of the equalities in the first line of \eqref{x_solution} with the corresponding $p_j$, $j=1,\hdots,n$, and taking the sum over all $j=1,\hdots,m$, one gets that
\begin{eqnarray*}
&& \sum_{j=1}^m p_j \overline{x}_j=\sum_{j=1}^m p_j e^{\frac{1}{\overline{\beta}_j}\left(v(D'_L)-\sum\limits_{i=1}^n\overline{\lambda}_ia_{ij}\right)-H}=e^{\frac{v(D'_L)}{\overline{\beta}}}\underbrace{\sum_{j=1}^m p_j e^{-\frac{\sum\limits_{i=1}^n\overline{\lambda}_ia_{ij}}{\overline{\beta}}-H}}_{=e^{-\frac{v(D'_L)}{\overline{\beta}}}}=1.
\end{eqnarray*}
Moreover, the first equality in \eqref{x_solution} is equivalent to
\begin{eqnarray*}
\ln \overline{x}_j=\frac{1}{\overline{\beta}}\left(v(D'_L)-\sum_{i=1}^n\overline{\lambda}_ia_{ij}\right)-H\quad \forall j=1,\hdots,m,
\end{eqnarray*}
and by multiplying in both sides with $\overline{x}_jp_j$ and taking the sum over all $j=1,\hdots,m$, it follows that
\begin{eqnarray*}
&&\sum_{j=1}^m\overline{x}_jp_j\ln \overline{x}_j=\frac{1}{\overline{\beta}}\left(\sum_{j=1}^m\overline{x}_jp_jv(D'_L)-\sum_{j=1}^m\sum_{i=1}^n\overline{\lambda}_ip_ja_{ij}\overline{x}_j\right)-\sum_{j=1}^m\overline{x}_jp_jH\\
&=&\frac{1}{\overline{\beta}}\underbrace{\left(v(D'_L)-\sum_{j=1}^m\sum_{i=1}^n\overline{\lambda}_ip_ja_{ij}\overline{x}_j\right)}_{=0}-H=-H,
\end{eqnarray*}
where $v(P')=\sum_{j=1}^m\sum_{i=1}^n\overline{\lambda}_ip_ja_{ij}\overline{x}_j$, i.e. $v(P')=v(D'_L)$.
\end{remark}

In \cite[Proposition 2.11]{AJP} the question of attaining the infimum in the definition of the entropic value-at-risk was discussed. The next remark gives a detailed insight regarding this question in the context of conjugate duality as we highlight the connection between this question, the conjugate function of the entropic value-at-risk, its biconjugate function, its dual representation and its subdifferential.

\begin{remark}\normalfont  Let $X^*\in L^q$, where $q\in (1, +\infty)$, $\alpha\in (0,1]$ and the function $g:L^q\rightarrow \overline{\mathbb{R}}$ be defined by
\begin{eqnarray*}
g(X^*)= \left\{ \begin{array}{ll} 0, \qquad
\textnormal{if } X^*\in L^q_+,~ \mathbb{E}(X^*\ln(X^*))\leq -\ln\alpha,~\mathbb{E}X^*=1,\\
+ \infty,  \textnormal{ otherwise}. \end{array} \right.
\end{eqnarray*}

It can be verified that $g$ is proper, convex and lower semicontinuous, and its conjugate function $g^*:L^p\rightarrow \overline{\mathbb{R}}$
(where $p\in (1, +\infty)$ is the \textit{H\"older conjugate} of $q$, i.e. it fulfills $1/p+1/q=1$) is given by
\begin{eqnarray}
g^*(X)=\sup_{X^*\in L^q_+,~\mathbb{E}X^*=1,\atop \mathbb{E}(X^*\ln(X^*))\leq -\ln\alpha} \langle X^*,X\rangle.\label{g_formula}
\end{eqnarray}
For a fixed $X\in L^p$, the negative of the right-hand side in \eqref{g_formula} is a convex optimization problem. By the corresponding weak Lagrange duality, one obtains
\begin{eqnarray}
g^*(X)\leq\inf_{\lambda\geq 0,~\gamma\in \mathbb{R}}\left\{\sup_{X^*\in L^q_+}\{\langle X^*,X\rangle-\lambda\mathbb{E}(X^*\ln(X^*))-\gamma\langle 1_{\Omega},X^*\rangle\}+\gamma-\lambda\ln\alpha\right\}.\label{g_formula1}
\end{eqnarray}
Moreover, applying \cite[Theorem 3.2.9.]{DVO} follows that the strong duality holds for this pair of primal-dual optimization problems, i.e. in \eqref{g_formula1} equality holds and the infimum in the right-hand side is attained. To see this, let $C=\mathbb{R}_+\times \{0\}$ and take $h:L^q\rightarrow \mathbb{R}^2\cup\{\infty_{\mathbb{R}^2_+}\}$,
\begin{eqnarray*}
h(X^*)= \left\{ \begin{array}{ll} (\mathbb{E}(X^*\ln(X^*))+\ln\alpha,\mathbb{E}X^*-1)^{\top},
\textnormal{ if } X^*\in L^q_+,\\
\infty_{\mathbb{R}^2_+},  \textnormal{ otherwise,} \end{array} \right.
\end{eqnarray*}
then it holds $h(\dom \langle X,\cdot\rangle\cap \dom h)=h(\dom h)=[\ln\alpha+\min_{X^*\in L^q_+}\mathbb{E}(X^*\ln(X^*)),+\infty)\times[-1,+\infty)$, where it is obvious that $\min_{X^*\in L^q_+}\mathbb{E}(X^*\ln(X^*))<0$, and therefore,
\begin{eqnarray*}
0_{\mathbb{R}^2}\in\ri\left(h(\dom h)+C \right)=\ri \left[\ln\alpha+\min_{X^*\in L^q_+}\mathbb{E}(X^*\ln(X^*)),+\infty\right)\times\ri[-1,+\infty).
\end{eqnarray*}
As the refined Slater constraint qualification is fulfilled, there is strong duality for the primal-dual pair of optimization problems mentioned above, hence
\begin{eqnarray*}
g^*(X)=\min_{\lambda\geq 0,~\gamma\in \mathbb{R}}\left\{\sup_{X^*\in L^q_+}\{\langle X^*,X\rangle-\lambda\mathbb{E}(X^*\ln(X^*))-\gamma\langle1_{\Omega},X^*\rangle\}+\gamma-\lambda\ln\alpha\right\}.
\end{eqnarray*}
In order to compute the quantity in the right-hand side of this equality we distinguish two cases. When $\lambda>0$, then we have by the interchangeability of minimization and integration as well as the considerations we made regarding the functions $\varphi$ and $q$ (while deriving the simplified formula of $(D_L)$ that
$$
\min_{\lambda>0,~\gamma\in \mathbb{R}}\left\{\sup_{X^*\in L^q_+}\{\langle X^*,X\rangle-\lambda\mathbb{E}(X^*\ln(X^*))-\gamma\langle1_{\Omega},X^*\rangle\}+\gamma-\lambda\ln\alpha\right\}=\min_{\lambda>0,~\gamma\in \mathbb{R}}\bigg\{\mathbb{E}$$
$$\left(\sup_{x^*\in\mathbb{R}_+}\{ Xx^* -\lambda x^*\ln(x^*)-\gamma x^*\}\right)+\gamma-\lambda\ln\alpha\bigg\}
= \min_{\lambda>0,~\gamma\in \mathbb{R}}\left\{\lambda\mathbb{E}\left(e^{\frac{X+\gamma-\lambda}{\lambda}}\right)+\gamma-\lambda\ln\alpha\right\}
$$
\begin{eqnarray}
=\min_{\lambda>0}\left\{\lambda\ln\left(\mathbb{E}\left(e^{\frac{X-\lambda+\lambda}{\lambda}}\right)\right)-\lambda\ln\alpha\right\}=
\min_{\lambda>0}\left\{\lambda\ln\left(\frac{\mathbb{E}\left(e^{\frac{X}{\lambda}}\right)}{\alpha}\right)\right\}.\label{g_formula12}
\end{eqnarray}
%, when $w=\mathbb{E}(e^{\frac{X-\lambda}{\lambda}})$ and $k=\lambda\ln\alpha$, attains its maximum at $\gamma=\lambda\ln w$, where $q(\lambda\ln w)=\lambda\ln\alpha-\lambda\ln(ew)$.

On the other hand, when $\lambda=0$ one has
\begin{eqnarray}
\min_{\gamma\in \mathbb{R}}\left\{\sup_{X^*\in L^q_+}\{\langle X^*,X\rangle-\gamma\langle1_{\Omega},X^*\rangle\}+\gamma\right\}=\min_{\gamma\in \mathbb{R},~X\leq \gamma}\gamma=\ess\sup(X).\label{g_formula13}
\end{eqnarray}
Bringing \eqref{g_formula12} and \eqref{g_formula13} together implies that
\begin{eqnarray*}
g^*(X)=\min\left\{\min_{\lambda>0}\left\{\lambda\ln\left(\frac{\mathbb{E}\left(e^{\frac{X}{\lambda}}\right)}{\alpha}\right)\right\}, \ess\sup(X)\right\}=\min\{\EVaR_{1-\alpha}(X),\ess\sup(X)\}.
\end{eqnarray*}
As $\EVaR_{1-\alpha}(X)\leq\ess\sup(X)$ for any $X\in L^p$ (see \cite[Remark 2.7]{AJP}), we have that $g^*(X)=\EVaR_{1-\alpha}(X)$. Further, by the Fenchel-Moreau theorem it holds $g=g^{**}$, which yields that $g$ is the conjugate function of $\EVaR_{1-\alpha}$ and the formula of the conjugate function of $g^*$ given in \eqref{g_formula} is a dual representation of the entropic value at risk similar to the one in \cite[Theorem 4.9]{AJP} (other dual representations of the entropic value-at-risk were proposed in \cite{AJM, AAJ}).

From here, one can also deduce a formula for the subdifferential of the entropic value at risk at $X\in L^p$, namely
\begin{eqnarray*}
&\partial \EVaR_{1-\alpha}(X)=\left\{X^*\in L^q:\EVaR_{1-\alpha}(X)+g(X^*)=\langle X^*,X\rangle\right\}&\\
&=\left\{X^*\in L^q_+: \EVaR_{1-\alpha}(X)=\langle X^*,X\rangle,~\mathbb{E}(X^*\ln(X^*))\leq -\ln\alpha,~\mathbb{E}X^*=1\right\}.&
\end{eqnarray*}
Finally, let us mention that we can use the formula of the conjugate function of $\EVaR_{1-\alpha}$ to study the following generalized portfolio optimization problem
\begin{eqnarray*}
\inf_{\lambda_i\geq 0,~ i=1,\hdots,n,\atop\lambda=(\lambda_1, \hdots, \lambda_n)^\top\in D,~\sum\limits_{i=1}^n\lambda_i=1}\left\{\EVaR_{1-\alpha}(G(\lambda,R)\right\}
\end{eqnarray*}
via conjugate duality, where $G(\lambda,R)$ is a random variable and $R$ is an $m$-dimensional vector of risk factors with a known probability distribution, characterizing the portfolio return (see \cite{AJF}). A further application of the formula of the conjugate function of $\EVaR_{1-\alpha}$ arises thus in machine learning with uncertainty in the input data, as considered in \cite{AJM}.
\end{remark}

\section{Proximal splitting method}\label{se4}

In this section we turn our attention to numerically determining an optimal solution to the optimization problem $(P)$ and its dual problem $(D_L)$. Due to the complicate structure of the problem $(D_L)$ a direct employment of (splitting) proximal point methods for solving it is currently not possible. However, one can solve the primal optimization problem $(P)$ by employing such an algorithm and then determine an optimal solution to $(D_L)$ via the necessary and sufficient optimality conditions derived in Section \ref{se2} (or Section \ref{se3} for the considered special case). For this purpose let us recast $(P)$ as an unconstrained optimization problem as follows. First, we introduce some additional variables $\tau_j$, $j=1,\hdots,m$,
\begin{eqnarray*}
(PU)&&\inf_{y\in \mathbb{R}^d,~x\in \mathbb{R}_+^m,~\tau_j\in \mathbb{R},~p_jx_j\ln x_j\leq \tau_j,~j=1,\hdots,m,\atop \sum\limits_{j=1}^m\tau_j\leq -H,~
~A_px+By\leqq h,~i=1,\hdots,n,~ \sum\limits_{j=1}^mp_jx_j=1}\left\{c_p^{\top}x+\sigma_D(y)\right\},
\end{eqnarray*}
and then we rewrite this optimization problem by using indicator functions that capture the constraints as
$$(PU)\ \inf_{(\tau,x,y)\in \mathbb{R}^m\times\mathbb{R}^m\times\mathbb{R}^d}\Bigg\{c_p^{\top}x+\sigma_D(y)+\delta_{\mathcal{C}_{1}}(x,y)+\delta_{\mathcal{C}_{2}}(x)
+\sum\limits_{j=1}^m
\delta_{\epi h_j}(x_j,\tau_j)+\delta_{\mathcal{B}}(\tau)\Bigg\},$$
where
\begin{eqnarray*}
&&\mathcal{C}_{1}=\left\{(x,y)\in\mathbb{R}^m\times\mathbb{R}^d:A_px+By\leqq h\right\},\ \mathcal{C}_{2}=\left\{x\in \mathbb{R}^m_+:\sum\limits_{j=1}^mp_jx_j=1\right\},\\
&&\mathcal{B}=\left\{\tau=(\tau_1,\hdots,\tau_m)^\top\in \mathbb{R}^m:\sum_{j=1}^m\tau_j\leq -H\right\},
\end{eqnarray*}
and
\begin{eqnarray*}
h_j:\mathbb{R}\rightarrow \mathbb{R},~h_j(x)= \begin{cases}
 p_jx\ln x, & \quad \text{if } x>0,\\
 0, & \quad \text{if } x=0,\\
  +\infty,  & \quad \text{if }  x<0,
  \end{cases}~j=1, \hdots, m.
\end{eqnarray*}
In order to employ a proximal splitting method for solving $(PU)$, one needs closed formulae of the projection operators onto the sets $\mathcal{C}_{1}$, $\mathcal{C}_{2}$ and $\mathcal{B}$, that were already established in the literature, as well as onto the epigraphs of the functions $h_j$, $j=1, \hdots, m$. The next statement presents formulae for the latter.

\begin{theorem}\label{ThmEpi} Let $p>0$ and
\begin{eqnarray*}
h:\mathbb{R}\rightarrow \mathbb{R},~h(x)= \begin{cases}
 px\ln x, & \quad \text{if } x\geq 0,\\
 %0, & \quad \text{if } x=0,\\
  +\infty,  & \quad \text{if }  x<0.
\end{cases}
\end{eqnarray*}
Then it holds for all $(x,\xi)\in \mathbb{R}\times \mathbb{R}$ that
\begin{eqnarray*}
P_{\epi h}(x,\xi)=\begin{cases}
 (x,\xi), & \quad \text{if } px\ln x\leq \xi,\\
 (0,\xi), & \quad \text{if } x\leq0\text{ and }\xi\geq 0,\\
 (e^{\bar z},p\bar ze^{\bar z}),  & \quad \text{otherwise,}
\end{cases}
\end{eqnarray*}
where $\bar z\in \mathbb{R}$ is the unique solution of the equation
\begin{eqnarray}\label{root}
(p^2z^2+p^2z+1)e^z-p\xi(z+1)-x=0
\end{eqnarray}
that fulfills $p\bar ze^{\bar z}\geq \xi$. More precisely, for $\xi\leq -p/e$ it holds $\bar z\in \mathbb{R}$, when $-p/e<\xi\leq 0$ and $x\leq 1/e$ one has $\bar z\in (-\infty,W_{-1}(\xi/p)]$, and when $-p/e<\xi\leq 0$ and $x> 1/e$ one has $\bar z\in [W(\xi/p),+\infty]$, while if $\xi>0$ it holds $\bar z\in [W(\xi/p),+\infty)$.
\end{theorem}

\begin{proof}
Note first that as the function $h$ is proper, convex and lower semicontinuous, its epigraph is nonempty, convex and closed, thus the projection onto it is well defined. Then, for any $(x,\xi), (\overline{y}, \overline{\theta})\in \mathbb{R}^2$ \cite[Proposition 6.46]{BCO} asserts that $(\overline{y}, \overline{\theta})=P_{\epi h}(x,\xi)$ if and only if $(x-\overline{y}, \xi-\overline{\theta})\in N_{\epi h}(\overline{y}, \overline{\theta})$.

Clearly, when $(x,\xi)\in \epi h$ it holds $P_{\epi h}(x,\xi)=(x,\xi)$. Assume further that $(x,\xi)\notin \epi h$, that can happen in exactly two situations, either when $x<0$ or if $px\ln x>\xi$ when $x\geq 0$.

%Let us first suppose that $(\overline{y},\overline{\theta})=P_{\epi h}(x,\xi)=(0,\xi)$ with $\xi>0$.

The first case we treat is when $x<0$ and $\xi \geq 0$. Then one has $(0, \xi)\in \epi h$ and
$$\sqrt{(x-0)^2+(\xi-\xi)^2} \leq \sqrt{(x-y)^2+(\xi-\theta)^2}\quad  \forall (y, \theta)\in \epi h,$$
that, combined, yield $P_{\epi h}(x,\xi)=(0,\xi)$.

Let us assume now that $P_{\epi h}(x,\xi)=(0,0)$, to see if other points $(x,\xi)$ besides those on the half-axe $(-\infty,0]\times \{0\}$ are projected on $\epi h$ in the origin. By \cite[Proposition 6.46]{BCO} one gets the equivalent formulation $(x,\xi)\in N_{\epi h}(0,0)$. If $\xi\neq 0$, then it holds by the definition of the normal cone that
$$xy+\xi \theta\leq 0 \text{ for all } (y,\theta)\in \epi h \Rightarrow xy+\xi py\ln y\leq 0\quad  \forall y> 0$$
\begin{eqnarray}\label{epih}
\Leftrightarrow x+\xi p\ln y\leq 0\quad  \forall y> 0 \Leftrightarrow  x\leq-\xi p\ln y\quad  \forall y> 0.
\end{eqnarray}
For $\xi>0$ we get by \eqref{epih} that
$$
x\leq-\xi p\lim_{y\rightarrow +\infty}\ln y=-\infty \Rightarrow x=-\infty,
$$
while when $\xi<0$ one gets
$$
x\leq-\xi p\lim_{y\rightarrow 0}\ln y=-\infty \Rightarrow x=-\infty.
$$
In both these cases one cannot have $(x,\xi)\in N_{\epi h}(0,0)$, therefore a contradiction has been reached and the only remaining possibility is $\xi=0$, when \eqref{epih} yields $x\leq 0$. Consequently $N_{\epi h}(0,0)=(-\infty,0]\times\{0\}$.

Note that in a similar manner one can show that $N_{\epi h}(0,\overline{\theta})=(-\infty,0]\times\{\overline{\theta}\}$ for any $\overline{\theta}> 0$, which guarantees that the only points having their projections onto $\epi h$ on the half-axe $\{0\}\times (0, +\infty)$ are the ones identified above that lie in the set $(-\infty,0)\times (0, +\infty)$.

Due to the fact that $\dom h=[0, +\infty)$, the only remaining possibility for some pair $(\overline{y},\overline{\theta})\in \mathbb{R}^2$ to be the projection of $(x,\xi)$
onto $\epi h$ is to have $\overline{y}>0$, that is equivalent to $(x,\xi)-(\overline{y},\overline{\theta})\in N_{\epi h}(\overline{y},\overline{\theta})$. Noticing that since $\overline{y}>0$ it holds $\partial h(\overline{y})=\{h'(y)\}$ where $h'(y)=p\ln y+p$, one gets for the normal cone that (see Section 2.1 in \cite{ABB})
\begin{eqnarray*}
N_{\epi h}(\overline{y},\overline{\theta})=\{(\lambda\gamma,-\lambda):\gamma\in \partial h(\overline{y})\text{ and }\lambda\geq 0\}
\end{eqnarray*}
and as $(x,\xi)-(\overline{y},\overline{\theta})\in N_{\epi h}(\overline{y},\overline{\theta})$, we have
\begin{eqnarray}
\left\{ \begin{array}{ll}x-\overline{y}=\lambda\gamma, \\
\xi-\overline{\theta}=-\lambda, \end{array} \right.\Leftrightarrow\left\{ \begin{array}{ll}x-\overline{y}=\lambda(p\ln \overline{y}+p), \\
\overline{\theta}=\xi+\lambda. \end{array} \right.\label{epih2}
\end{eqnarray}
If $\lambda=0$ one has by \eqref{epih2} that $\overline{y}=x$ and $\overline{\theta}=\xi$, i.e. $(x,\xi)\in \epi h$, which is a contradiction to our assumption from the beginning of this proof. Therefore, $\lambda>0$. Moreover, notice that by \cite[Proposition 28.28]{BCO} it holds $\overline{\theta}=p\overline{y}\ln\overline{y}$. Further, by using the Lambert $W$ function and the property $z=W(z)e^{W(z)}$ for all $z\in \mathbb{R}$ (see \cite{BLI}) we have that
\begin{eqnarray}
\ln\overline{y}=\frac{\overline{\theta}}{p\overline{y}}\Leftrightarrow \overline{y}=e^{\frac{\theta}{p\overline{y}}}\Leftrightarrow \frac{\overline{\theta}}{p}=\frac{\overline{\theta}}{p\overline{y}}e^{\frac{\theta}{p\overline{y}}}\Leftrightarrow W\left(\frac{\overline{\theta}}{p}\right)=\frac{\overline{\theta}}{p\overline{y}}\Leftrightarrow \overline{y}=\frac{\frac{\overline{\theta}}{p}}{W\left(\frac{\overline{\theta}}{p}\right)}\Leftrightarrow\overline{y}=e^{W\left(\frac{\overline{\theta}}{p}\right)},\label{epih3}
\end{eqnarray}
hence, $\overline{\theta}=p\overline{y}\ln \overline{y}=pze^z$ where $z:=W(\overline{\theta}/p)$. Combining \eqref{epih2} and \eqref{epih3} reveals that
\begin{eqnarray*}
&&x-e^z=\lambda(pz+p)\Leftrightarrow \lambda=\frac{x-e^z}{p(z+1)}\Leftrightarrow \frac{\lambda+\xi}{p}=\frac{x-e^z+p\xi (z+1)}{p^2(z+1)}\notag\\
&\Leftrightarrow& \frac{\overline{\theta}}{p}=\frac{x-e^z+p\xi (z+1)}{p^2(z+1)}\Leftrightarrow W\left(\frac{\overline{\theta}}{p}\right)e^{W\left(\frac{\overline{\theta}}{p}\right)}=\frac{x-e^z+p\xi (z+1)}{p^2(z+1)}\notag\\
&\Leftrightarrow& ze^z=\frac{x-e^z+p\xi (z+1)}{p^2(z+1)}\Leftrightarrow (p^2z^2+p^2z+1)e^z-p\xi(z+1)-x=0,\label{epih4}
\end{eqnarray*}
which, taking into consideration \eqref{epih2} verifies the last case.

Finally, it remains to show the uniqueness of the solution of \eqref{root} when $(x, \xi) \notin \epi h \cup (-\infty, 0]\times [0, +\infty)$. For this purpose, we define the function $F:\mathbb{R}\rightarrow \mathbb{R}$ by $$F(z)=(p^2z^2+p^2z+1)e^z-p\xi(z+1)-x=p(z+1)(pze^z-\xi)+e^z-x,$$
whose derivative is
$$
F'(z)=(2p^2z+p^2)e^z+(p^2z^2+p^2z+1)e^z-p\xi=e^zp^2(z+1)^2+e^z+p(pze^z-\xi),\ z\in \mathbb{R}.
$$
One has
\begin{equation}\label{eq21}
\lim_{z\rightarrow +\infty}F(z)= p^2\lim_{z\rightarrow +\infty}(z^2e^z) +p\lim_{z\rightarrow +\infty}z(pe^z-\xi)+\lim_{z\rightarrow +\infty} e^z-p\xi-x=+\infty,
\end{equation}
and
\begin{equation}\label{eq22}
\lim_{z\rightarrow -\infty}F(z) = p^2\lim_{z\rightarrow -\infty}(z^2e^z)+p\lim_{z\rightarrow -\infty}z(pe^z-\xi)+\lim_{z\rightarrow -\infty} e^z-p\xi-x=
\left\{
\begin{array}{cc}
 \sgn(\xi) (+\infty), & \mbox{if } \xi\neq 0,\\
  -x, & \mbox{if } \xi=0,
\end{array}\right.
\end{equation}
where $\sgn (\cdot)$ is the usual \textit{signum function}, as well as
$\lim_{z\rightarrow +\infty}F'(z)=+\infty$ and $\lim_{z\rightarrow -\infty}F'(z)=-p\xi$.
Moreover, $F''(z) = e^z\big(p^2z^2 + 5p^2z + 4p^2+1\big)$, $z\in \mathbb{R}$. Analyzing this function, one notices that $F''(z)>0$ when $0<p<2/3$, $F''(z)\geq 0$ when $p=2/3$, while when
$p>2/3$ $F''$ takes nonnegative values outside the interval $\big( (-5p-\sqrt{9p^2-4})/(2p), (-5p+\sqrt{9p^2-4})/(2p)\big)$, on which it is negative.

This shows that when $p\in (0, 2/3]$ the function $F'$ is always positive when $\xi<0$, in which situation $F$ is strictly increasing and, consequently,
the solution $\bar z\in \mathbb{R}$ to the equation \eqref{root} is unique and it can be shown that it verifies $p\bar ze^{\bar z}\geq \xi$. However, in case $p>2/3$ the calculations become more complicated and a different approach is required. Note also that, depending on the values of $p$, $x$ and $\xi$ (as considered above) equation \eqref{root} has at least one and up to four solutions.

%Then $F'(z)=e^z \varphi (z) -p\xi$ for all $z\in \mathbb{R}$, where $\varphi : \mathbb{R} \rightarrow \mathbb{R}$, $\varphi(z)=p^2(z+1)^2+1+p^2z$, $z\in \mathbb{R}$. Analyzing %the function $\varphi$ one can notice that it is everywhere positive when $p\in (0, 2/\sqrt{5})$, is nonnegative when $p=2/\sqrt{5}$, while when $p>2/\sqrt{5}$ it is nonnegative %except for an interval where it takes negative values.

In the following, we distinguish six cases identified by the possible values of $x$ and $\xi$, when $p>0$.

$(i)$ When $\xi\leq-p/e$, noting that
\begin{eqnarray}\label{eq19}
\min_{z\in \mathbb{R}}\{pze^z\}=-\frac pe,
\end{eqnarray}
it follows $F'(z)> 0$ $\forall z\in \mathbb{R}$, so $F$ is strictly monotone increasing in this situation, consequently \eqref{eq21} and \eqref{eq22} yield that the solution $\bar z\in \mathbb{R}$ to the equation \eqref{root} is unique. By \eqref{eq19}, it also verifies $p\bar ze^{\bar z}\geq \xi$.

$(ii)$ When $-p/e<\xi\leq0$ and $x<0$, one has
\begin{eqnarray}\label{eq20}
F'(z)=e^zp^2(z+1)^2+e^z+p(pze^z-\xi)>0 \quad  \forall z\in \mathbb{R} \text{ such that } pze^z\geq \xi.
\end{eqnarray}
Since $\lim_{z\rightarrow +\infty} pze^z = +\infty$ and $\lim_{z\rightarrow -\infty} pze^z = 0$, taking into consideration \eqref{eq19} one can show that the equation $pze^z=\xi$ has two solutions, say $\tilde{z}, \hat{z} \in \mathbb{R}$ with $\tilde{z}<-1 < \hat{z}<0$. Via \eqref{eq20} one concludes that $F$ is strictly increasing on $(-\infty, \tilde{z}]$ and on $[\hat{z}, +\infty)$, and can decrease only on a subinterval of $(\tilde{z}, \hat{z})$. Since $F(\tilde{z})>0$ and $F(\hat{z})>0$, by \eqref{eq21} and \eqref{eq22} follows the uniqueness of the solution $\bar z<\tilde{z}$ to \eqref{root} under the current hypotheses. As $p\bar ze^{\bar z}\geq \xi$ can be equivalently written as $\bar z=W_{-1}(\bar ze^{\bar z})\leq W_{-1}(\xi/p)$, one concludes $\bar z\in(-\infty,W_{-1}(\xi/p)]$.

$(iii)$ When $-p/e<\xi\leq0$ and $0\leq x\leq 1/e$, one determines again $\tilde{z}, \hat{z} \in \mathbb{R}$. Then $px\ln x>p\tilde{z}e^{\tilde{z}}=p\hat{z}e^{\hat{z}}=\xi$, which implies by monotonicity (as the function $h$ is strictly decreasing on $[0, 1/e]$) that $x<e^{\tilde{z}}<e^{\hat{z}}$, therefore $F(\tilde{z})>0$ and $F(\hat{z})>0$. The conclusion from the previous case can be drawn in this situation, too.

$(iv)$ When $-p/e<\xi\leq0$ and $1/e<x\leq 1$, one determines again $\tilde{z}, \hat{z} \in \mathbb{R}$. Then $px\ln x>p\overline{z}e^{\overline{z}}=\xi$, which implies by monotonicity (as the function $h$ is strictly increasing on $(1/e, +\infty)$) that $e^{\tilde{z}}<e^{\hat{z}}<x$, therefore $F(\tilde{z})<0$ and $F(\hat{z})<0$. Then equation \eqref{root} has a unique solution $\bar z>\hat{z}$ such that $\xi \leq p\bar ze^{\bar z}$. As $p\bar ze^{\bar z}\geq \xi$ can be equivalently written in this case as $\bar z=W(\bar ze^{\bar z})\geq W(\xi/p)$, one concludes $\bar z\in[W(\xi/p),+\infty)$.

$(v)$ When $-p/e<\xi\leq0$ and $x> 1$, one determines again $\tilde{z}, \hat{z} \in \mathbb{R}$. As $e^{\tilde{z}}<e^{\hat{z}}<x$, $F(\tilde{z})<0$ and $F(\hat{z})<0$, thus the conclusion from the previous case can be drawn in this situation, too.

$(vi)$ When $\xi \geq 0$, the only situation not covered yet in this analysis is when $px\ln x>\xi\geq 0$, that yields $x>1$. Moreover, one can show that the equation $pze^z=\xi$ has only one solution, say $\hat{z}\in \mathbb{R}$ with $\hat{z}>0$. Then again $x>e^{\hat{z}}$, thus $F(\hat{z})<0$ and the uniqueness of the solution of \eqref{root} that verifies $p\bar ze^{\bar z}\geq \xi$, i.e. lies in the interval $[W(\xi/p),+\infty)$, follows analogously to the previous case.
\end{proof}

\begin{remark}\normalfont  As seen in the proof of Theorem \ref{ThmEpi} the function $h$ is not subdifferentiable at the origin. It was shown that the normal cone of $\epi h$ at $(0,0)$ is $N_{\epi h}(0,0)=(-\infty,0]\times\{0\}$ and this means that the only supporting hyperplane to the epigraph of $h$ at the origin is vertical.
\end{remark}

The next statement provides a concretization of Theorem \ref{ThmEpi} for the situation when $0<p\leq 1$.

\begin{corollary} \label{co4} Let $0<p\leq 1$, then it holds for all $(x,\xi)\in \mathbb{R}\times \mathbb{R}$ that
\begin{eqnarray*}
P_{\epi h}(x,\xi)=\begin{cases}
 (x,\xi), & \quad \text{if } px\ln x\leq \xi,\\
 (0,\xi), & \quad \text{if } x\leq0\text{ and }\xi\geq 0,\\
 (e^{\bar z},p\bar ze^{\bar z}),  & \quad \text{otherwise,}
\end{cases}
\end{eqnarray*}
where $\bar z\in \mathbb{R}$ is the unique solution of \eqref{root} that fulfills $p\bar ze^{\bar z}\geq \xi$. More precisely, exactly one of the following six cases holds\begin{enumerate}
  \item[(i)] $\xi\leq -\frac{p}{e}$ and $x<0$, then $\overline{z}\in \mathbb{R}$ is unique and $\overline{z}\in(-\infty,0]$;
  \item[(ii)] $\xi\leq -\frac{p}{e}$ and $x\geq 0$, then $\overline{z}\in \mathbb{R}$ is unique and $\overline{z}\in (-\infty,x]$;
    \item[(iii)] $-\frac{p}{e} <\xi \leq 0$ and $x<0$, then $\overline{z}\in \mathbb{R}$ is unique and $\bar z\in (-\infty,W_{-1}(\xi/p)]$;
  \item[(iv)] $-\frac{p}{e} <\xi \leq 0$ and $0\leq x\leq \frac{1}{e}$, then $\overline{z}\in \left[W_{-1}\left(\frac{\xi}{p}\right)+\ln(x)-1,W_{-1}\left(\frac{\xi}{p}\right)\right]$;
  \item[(v)] $-\frac{p}{e} <\xi \leq 0$ and $x> \frac{1}{e}$, then $\overline{z}\in\left[W\left(\frac{\xi}{p}\right),W\left(\frac{\xi}{p}\right)+\ln(x)+2\right]$;
  \item[(vi)] $\xi>0$, then $x>1$ and $\overline{z}\in \left[W\left(\frac{\xi}{p}\right),W\left(\frac{\xi}{p}\right)+\ln(x)+1\right]$.
\end{enumerate}
\end{corollary}

\begin{proof} Theorem \ref{ThmEpi} implies immediately the formulae for $P_{\epi h}$, the equation and the situation when $\xi\leq -\frac{p}{e}$ and $x\in \mathbb{R}$, then $\overline{z}\in \mathbb{R}$. When $x<0$, then $F(0)=1-p\xi-x>0$ and together with the fact that $F'(z)>0$ for all $z\in \mathbb{R}$ one gets case $(i)$.

If $x\geq 0$, then $F(x)=p(x+1)(pxe^x-\xi)+e^x-x>0$ and again by using the fact that $F'(z)>0$ for all $z\in \mathbb{R}$ follows case $(ii)$.

When $-p/e <\xi \leq 0$ and $x<0$, then it is mentioned in the proof of Theorem \ref{ThmEpi} that $F$ can decrease only on the subinterval of $(\tilde{z},\hat{z})$ where $\tilde{z}<-1<\hat{z}<0$. If one takes any $\breve{z}\in(\tilde{z},\hat{z})$ with $\breve{z}<-1$, then $p\breve{z}e^{\breve{z}}<\xi$ and
\begin{eqnarray*}
F(\breve{z})=\underbrace{p(\breve{z}+1)}_{<0}\underbrace{(p\breve{z}e^{\breve{z}}-\xi)}_{<0}+\underbrace{e^{\breve{z}}-x}_{>0}>0.
\end{eqnarray*}
Further, for any $\breve{z}\in(\tilde{z},\hat{z})$ with $\breve{z}>-1$ it also holds $p\breve{z}e^{\breve{z}}<\xi$ and
\begin{eqnarray*}
F(\breve{z})=\underbrace{(p^2\breve{z}^2+p^2\breve{z}+1)}_{>0, \text{ since } 0<p\leq 1}e^{\breve{z}}\underbrace{-p\xi(\breve{z}+1)-x}_{>0}>0.
\end{eqnarray*}
These two cases show that $F$ is strictly positive on the subinterval for all $z\in(\tilde{z},\hat{z})$ such that $\overline{z}$ is the unique solution of \eqref{root} on the whole space $\mathbb{R}$ and $\bar z\in (-\infty,W_{-1}(\xi/p)]$, i.e. case $(iii)$ is true.

Moreover, if $-p/e <\xi \leq 0$ and $0\leq x\leq 1/e$, then $\overline{z}\in(-\infty,W_{-1}(\xi/p)]$. From \eqref{eq21} and \eqref{eq22} one has that $\lim_{z\rightarrow +\infty}F(z)=+\infty$ as well as $\lim_{z\rightarrow -\infty}F(z)=-\infty$, hence, $F(W_{-1}(\xi/p))>0$. As $W_{-1}(\xi/p)<0$ and $\ln(x)<0$, it follows $W_{-1}(\xi/p)+\ln(x)-1<W_{-1}(\xi/p)<0$ such that
%\begin{eqnarray*}
$$\underbrace{p(W_{-1}(\xi/p)+\ln(x))}_{<0}\underbrace{(p(W_{-1}(\xi/p)+\ln(x)-1)e^{W_{-1}(\xi/p)+\ln(x)-1}-\xi)}_{\geq0, \text{ since }W_{-1}(\xi/p)+\ln(x)-1\in (-\infty,W_{-1}(\xi/p)]}
+\!\!\!\!\!\underbrace{x(e^{W_{-1}(\xi/p)-1}-1)}_{\leq 0, \text{ as } x\geq 0,\ W_{-1}(\xi/p)-1<0}\!\!\!\!\! <0,$$
%\end{eqnarray*}
i.e. $F(W_{-1}(\xi/p)+\ln(x)-1)<0$, which verifies $(iv)$.

Now, if $-p/e<\xi\leq 0$ and $x>1/e$, then $\overline{z}\in [W(p/e),+\infty)$. Using again \eqref{eq21} and \eqref{eq22} reveals that $F(W(\xi/p))<0$ and since $W(\xi/p)\geq-1$, it holds that $W(\xi/p)+\ln(x)+2>0$ (notice also that $x>1/e$, i.e. $\ln(x)>-1$) and
%\begin{eqnarray*}
$$\underbrace{p(W(\xi/p)+\ln(x)+3)}_{>0}\underbrace{(p(W(\xi/p)+\ln(x)+2)e^{W(\xi/p)+\ln(x)+2}-\xi)}_{>0, \text{ since }W(\xi/p)+\ln(x)+2\in [W(\xi/p),+\infty)}
+\underbrace{x(e^{W(\xi/p)+2}-1)}_{\geq 0}>0,$$
%\end{eqnarray*}
which means that $F(W(\xi/p)+\ln(x)+2)>0$ and thus, case $(v)$ is true.

Finally, let $\xi>0$, then $x>1$ and $\overline{z}\in [W(\xi/p),+\infty)$. As $px\ln(x)>\xi$, it holds $W(x\ln x)=\ln x>W(\xi/p)$, which implies that $x>e^{W(\xi/p)}$. Furthermore, since $\lim_{z\rightarrow+\infty}F(z)=+\infty$ and
%\begin{eqnarray*}
$$F(W(\xi/p))=\underbrace{p(W(\xi/p)+1)\left(p\frac{\xi}{p}-\xi\right)}_{=0}+e^{W(\xi/p)-x}<0,$$
%\end{eqnarray*}
we have that $W(\xi/p)+\ln(x)+1>W(\xi/p)$ and hence,
%\begin{eqnarray*}
$$\underbrace{p(W(\xi/p)+\ln(x)+2)}_{>0}\underbrace{(p(W(\xi/p)+\ln(x)+1)e^{W(\xi/p)+\ln(x)+1}-\xi)}_{>0, \text{ since }W(\xi/p)+\ln(x)+1\in [W(\xi/p),+\infty)}
+\underbrace{x(e^{W(\xi/p)+1}-1)}_{\geq 0, \text{ since } W\left(\xi/p\right)\geq -1}>0,$$
%\end{eqnarray*}
i.e. $F(W(\xi/p)+\ln(x)+1)>0$, which verifies $(vi)$.
\end{proof}

\begin{remark}\normalfont
Note that the hypothesis $0<p\leq 1$ was relevant only in the case $(iii)$ of Corollary \ref{co4}, while the other five situations described there hold also in the general case where $p>0$.
\end{remark}

\section{Numerical experiments}\label{se5}

In order to demonstrate how portfolio optimization problems where the risk is modelled by the entropic value-at-risk can be numerically solved by parallel splitting algorithms, we make use of the results derived in Section \ref{se3}. More precisely, we determine optimal solutions to the dual problem $(D'_L)$ obtained there by employing the Douglas-Rachford parallel splitting algorithm provided in \cite{BCO} for solving $(P')$ (recast as an unconstrained optimization problem as
\begin{eqnarray*}
(PU')&& \inf_{(t,\tau,x,y)\in \mathbb{R}\times\mathbb{R}^m\times\mathbb{R}^m\times\mathbb{R}^d}
\left\{t+\sigma_{\widetilde{D}}(y) +\delta_{\mathcal{C}_{2}}(x)+\delta_{\mathcal{C}_{3}}(x,y,t)+ \sum\limits_{j=1}^m \delta_{\epi h_j}(x_j,\tau_j) +\delta_{\mathcal{B}}(\tau) \right\}\\
\end{eqnarray*}
where
$\mathcal{C}_{3}=\left\{(x,y, t)\in\mathbb{R}^m\times\mathbb{R}^d\times\mathbb{R}: \sum\limits_{j=1}^mp_ja_{ij}x_j-y_i\leq t,\ i=1,\hdots,d\right\}$,
and $\mathcal{B}$, $\mathcal{C}_{2}$ and $h_j$, $j=1, \hdots, m$, are defined as in Section \ref{se4}), and then exploiting the optimality conditions provided in Theorem \ref{theoremChapter1:3.2} in order to determine an optimal solution to its Lagrange dual problem, that is a portfolio optimization problem where the risk is quantified by entropic value-at-risk. The performance of this procedure is compared with the one of the primal dual interior point algorithm EVaR-PD used in \cite[Section 4.2]{AJF} for solving the same portfolio optimization problem. In the tables and figures below, the values regarding the parallel splitting algorithm correspond to the performance of the mentioned algorithm from \cite{BCO} when solving $(PU')$, to which the effort of solving the optimality conditions system in order to recover an optimal solution to the considered portfolio optimization problem is added. On the other hand, the graphs corresponding to the primal-dual interior point method are ``mirrored'' in order to reflect a maximization problem, as is the dual to $(PU')$.
%We assess the efficiency of our proposed parallel splitting algorithm for solving the considered entropy portfolio optimization problem.

We choose as stopping criteria for our iterative method the values $\epsilon_1=10^{-3}$, $\epsilon_2=10^{-4}$, $\epsilon_3=10^{-6}$ and $\epsilon_4=10^{-8}$, where for $i=1,2,3,4$, $\epsilon_i$ stands for the maximal bounds between the objective function values given by two consecutive iterations. The origin of the corresponding spaces were taken as starting points. Our computational experiments were performed in \textsc{MATLAB} on a desktop computer equipped with an 11th Gen Intel Core i7-1185G7 processor (4 cores, 8 threads, base frequency 3.00 GHz, up to 4.80 GHz turbo), 16 GB of RAM, and a 64-bit operating system. The numerical examples used in this study were not based on real financial data, the involved data being randomly generated for simulation purposes, as our scope is to show the applicability of our method, as an alternative to the one proposed in \cite{AJF}. Nonetheless, the proposed approach could be readily applied to real data in future analyses.

\begin{example}\label{exn10m10}
	We first consider an example where $d=m=10$, $D=\mathbb{R}^n$, and the stepsizes $\nu$ and parameters $\mu$ take various values. The best performances of our tests are illustrated in Tables \ref{tab:Entropyn10m10e1}-\ref{tab:Entropyn10m10e4}.

	\begin{table}[h!]
		\centering
		\begin{tabular}{lcc}
			\hline %\toprule
			Algorithm&\textbf{PS} ($\mu=1.65$, $\nu=1.4$)&\textbf{PD} \\
			\hline %\midrule
			Number of iterations & $646$ & $351$ \\
			CPU time (seconds) & $1.0744$ & $3.580059$ \\
			Objective value & $2.299084$ &$2.296731$\\
			\hline %\bottomrule
		\end{tabular}
		\caption{Comparison of splitting (PS) and primal-dual interior point (PD) methods when $d=10$ and $m=10$ for $\epsilon_1=10^{-3}$}
		\label{tab:Entropyn10m10e1}
	\end{table}
	
	\begin{table}[h!]
		\centering
		\begin{tabular}{lcc}
			\hline %\toprule
			Algorithm&\textbf{PS} ($\mu=1.65$, $\nu=1.4$)&\textbf{PD} \\
			\hline %\midrule
			Number of iterations & $924$ & $352$ \\
			CPU time (seconds) & $1.3984$ & $3.785709$ \\
			Objective value & $2.297589$ &$2.296807$\\
			\hline %\bottomrule
		\end{tabular}
		\caption{Comparison of splitting (PS) and primal-dual interior point (PD) methods when $d=10$ and $m=10$ for $\epsilon_2=10^{-4}$}
		\label{tab:Entropyn10m10e2}
	\end{table}
	
	\begin{table}[h!]
		\centering
		\begin{tabular}{lcc}
			\hline %\toprule
			Algorithm&\textbf{PS} ($\mu=1.65$, $\nu=1.4$)&\textbf{PD} \\
			\hline %\midrule
			Number of iterations &$1369$ & $354$ \\
			CPU time (seconds) & $2.2137$ & $4.394282$ \\
			Objective value& $2.296820$ & $2.296826$ \\
			\hline %\bottomrule
		\end{tabular}	
\caption{Comparison of splitting (PS) and primal-dual interior point (PD) methods when $d=10$ and $m=10$ for $\epsilon_3=10^{-6}$} \label{tab:Entropyn10m10e3}
	\end{table}
	
	\begin{table}[h!]
		\centering
		\begin{tabular}{lcc}
			\hline %\toprule
			Algorithm&\textbf{PS} ($\mu=1.65$, $\nu=1.4$)&\textbf{PD} \\
			\hline %\midrule
			Number of iterations &$1939$ & $357$ \\
			CPU time (seconds) & $2.9496$ & $4.421544$ \\
			Objective value& $2.296827$ & $2.296827$ \\
			\hline %\bottomrule
		\end{tabular}	
\caption{Comparison of splitting (PS) and primal-dual interior point (PD) methods when $d=10$ and $m=10$ for  $\epsilon_4=10^{-8}$} \label{tab:Entropyn10m10e4}
	\end{table}
	
	The parallel splitting proximal point method delivered the optimal solution to $(PU')$ consisting of
$$\left\{	
\begin{array}{l}
		\bar{x}=\left(0.9463, 0.9369, 1.3343, 3.3603, 0.0500, 0.8699, 0.0675, 0.5003, 0.1664, 1.7321\right)^\top\\
\bar{y}=0_{\mathbb{R}^{10}}\\		
\bar{\tau}= \left(-0.0011, -0.0099, 0.0415, 0.4953, -0.0286, -0.0029, -0.0155, -0.0389, -0.0227, 0.0937\right)^\top\\
\bar{t}= 4.7414.
	\end{array}
\right.$$

	%We can deduce the value for $y$ considering the choice of the set $D$. Indeed as $D=\mathbb{R}^n$, since the term $\sigma_{D}(y)$ appears in the objective function, we have $.\\

Next, we determine an optimal asset allocation via the portfolio optimization problem $(D'_L)$ by solving the system of equations consisting of the optimality conditions provided in Theorem \ref{theoremChapter1:3.2}. The optimal solution to $(D'_L)$ determined in this way is
$$\left\{	
\begin{array}{l}
\bar{\beta}=3.6973\\
\bar{\lambda} = \left(0.0000, 0.0516, 0.0000, 0.5575, 0.0000, 0.0000, 0.0000, 0.0000, 0.3909, 0.0000\right)^\top,
	\end{array}
\right.$$
while $v(PU')=v(D'_L)=2.296827$.	

Although the primal-dual interior point method from \cite{AJF} requires fewer iterations (see Tables  \ref{tab:Entropyn10m10e1}--\ref{tab:Entropyn10m10e4}), the parallel splitting method demonstrates faster computational times for all considered error tolerances. For instance, the parallel splitting algorithm converged three times faster than the primal-dual interior point method from \cite{AJF} for $\epsilon_1=10^{-3}$ ($1.0744$ vs $3.580059$ seconds) and almost twice as fast when $\epsilon_3=10^{-6}$ ($2.2137$ vs $4.394282$ seconds), while both methods achieved comparable objective values.

The convergence of objective function values and optimal solutions to the portfolio optimization problem $(D'_L)$ is illustrated in Figure \ref{fig:n10m10cvobjvalEVaR-PDPS} and Figure \ref{fig:n10m10cvoptsolEVaR-PDPS}, respectively. It is important to note that the objective value of the primal optimization problem solved using the primal-dual (PD) interior point method, is equal to the negative of the objective value of the dual problem solved using the parallel splitting (PS) algorithm. To facilitate a direct comparison of the performances the two methods, the plotted values for the PD algorithm's results correspond to the negated primal objective values which are numerically equal to the dual objective values obtained by the PS algorithm. The apparent increase in the left-hand plots (PD) versus decrease in the right-hand plots (PS) thus reflects the natural monotonicity of the respective underlying optimization problems under their original sign conventions, rather than any discrepancy between the algorithms. The tables below report these adjusted values for clarity.
	
	\begin{figure}[t!]
		\centering
		\begin{subfigure}[t]{.45\textwidth}
			\centering
			\includegraphics[width=1\linewidth]{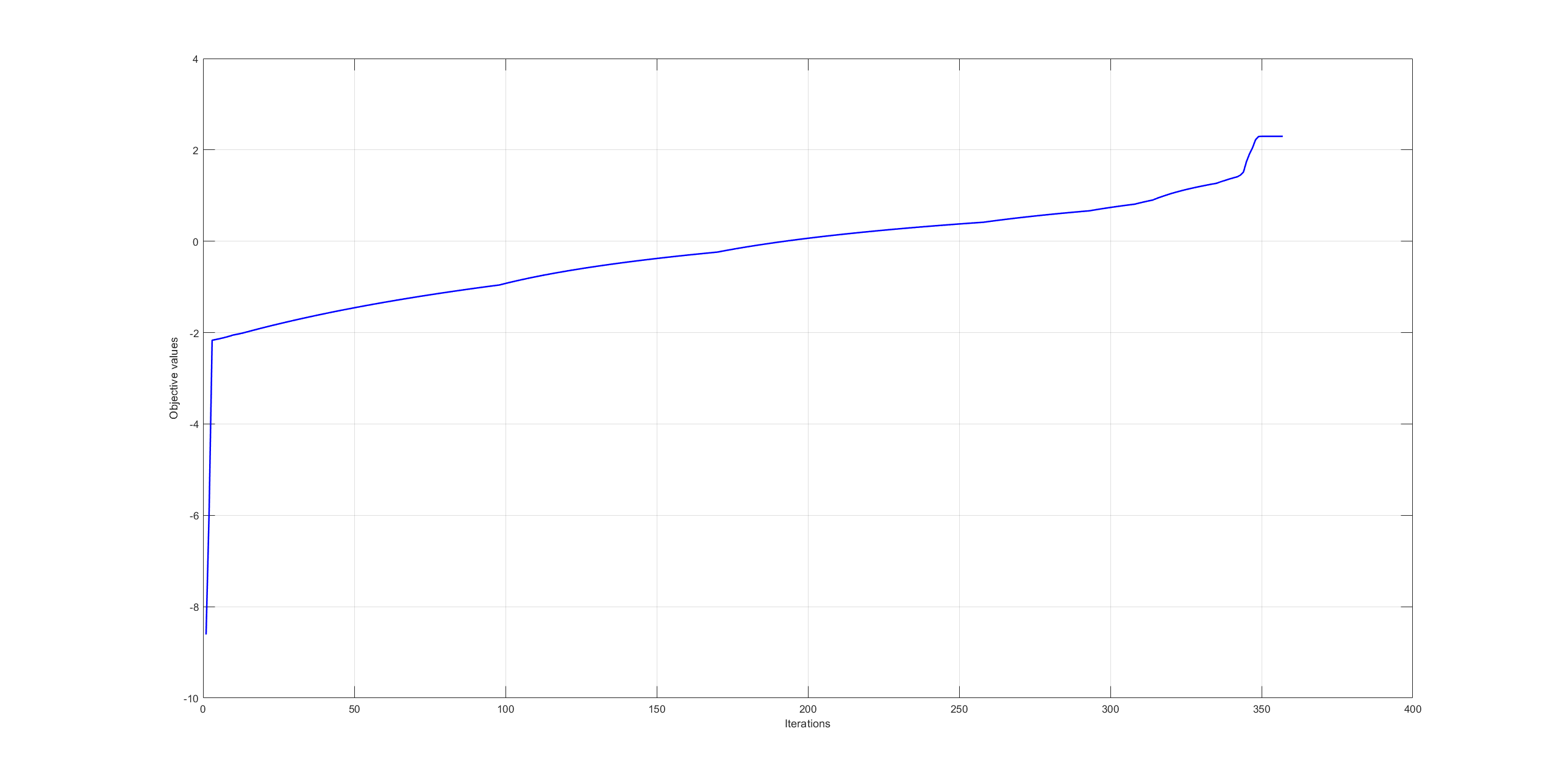}
			\caption{Objective value via primal-dual interior point}
			\label{fig:n10m10cvobjvalEVaR-PD}
		\end{subfigure}
		\begin{subfigure}[t]{.45\textwidth}
			\centering
			\includegraphics[width=1\linewidth]{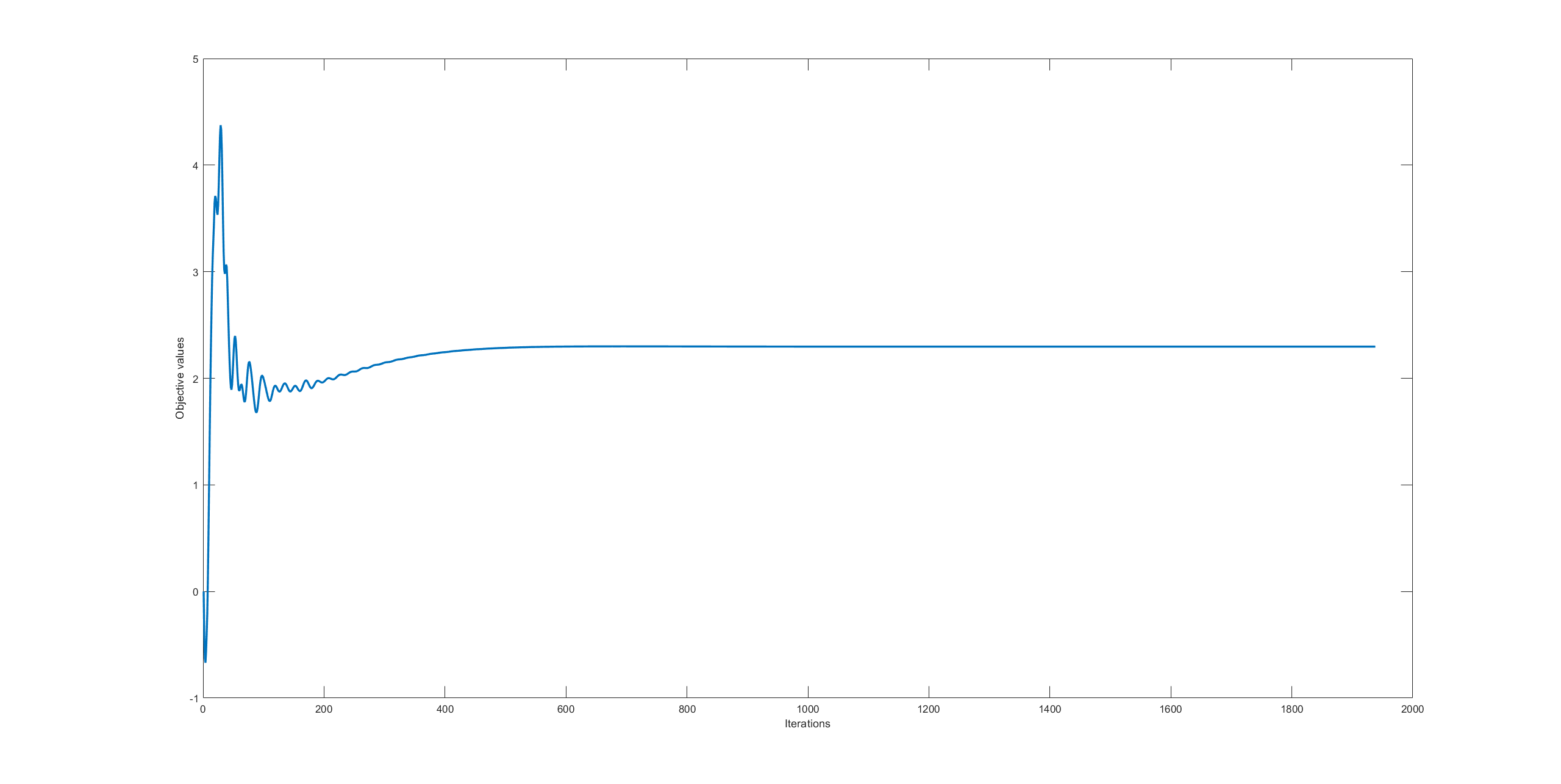}
			\caption{Objective value via parallel splitting}
			\label{fig:n10m10cvobjvalPS}
		\end{subfigure}
		\caption{Convergence of the objective value of problem $(D'_L)$ ($d=10$, $m=10$) via primal-dual interior point and parallel splitting methods}
		\label{fig:n10m10cvobjvalEVaR-PDPS}
	\end{figure}
	
	\begin{figure}
		\centering
		\begin{subfigure}{.45\textwidth}
			\centering
			\includegraphics[width=1\linewidth]{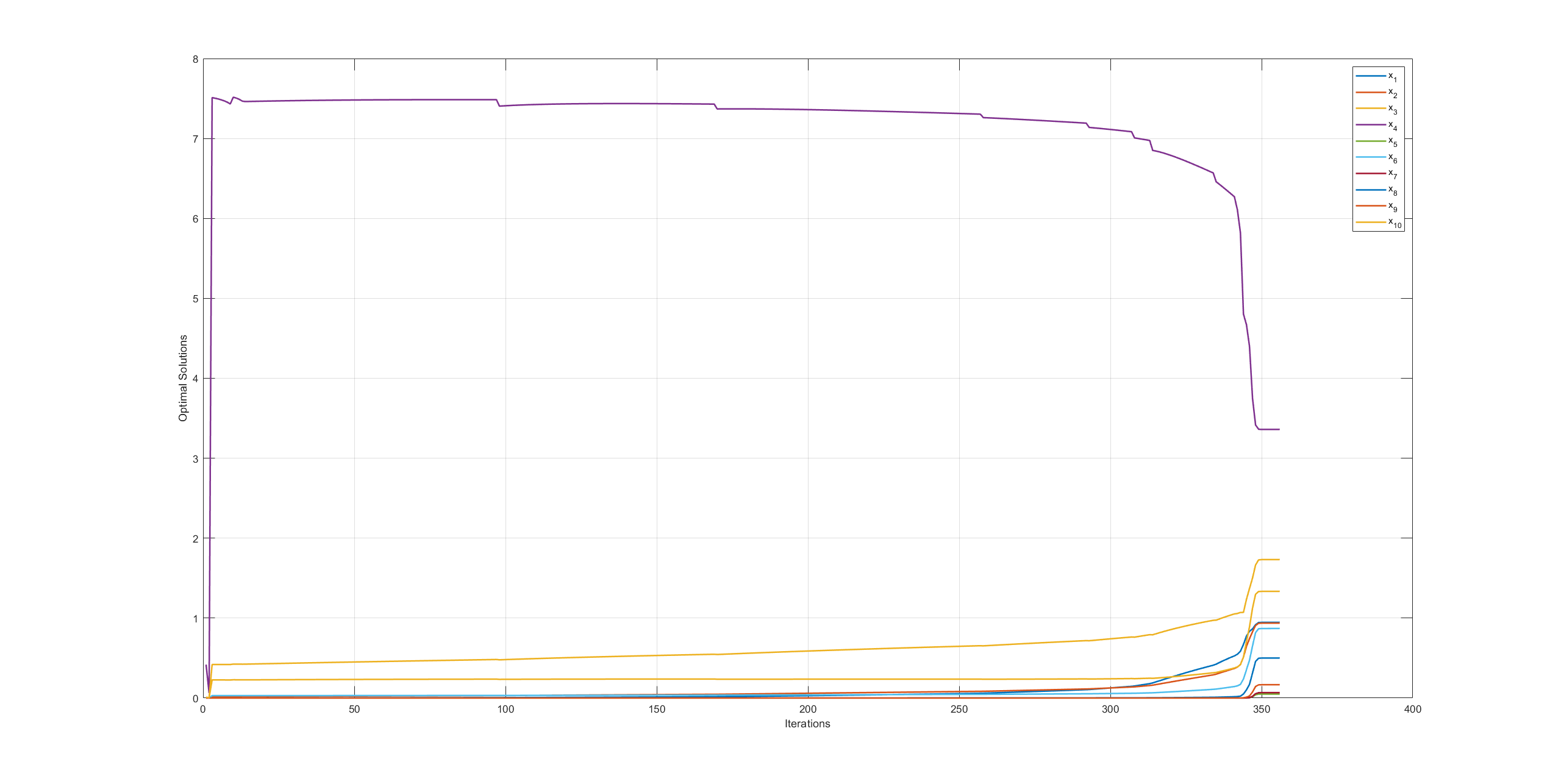}
			\caption{Optimal solution via primal-dual interior point}
			\label{fig:n10m10cvoptsolEVaR-PD}
		\end{subfigure}
		\begin{subfigure}{.45\textwidth}
			\centering
			\includegraphics[width=1\linewidth]{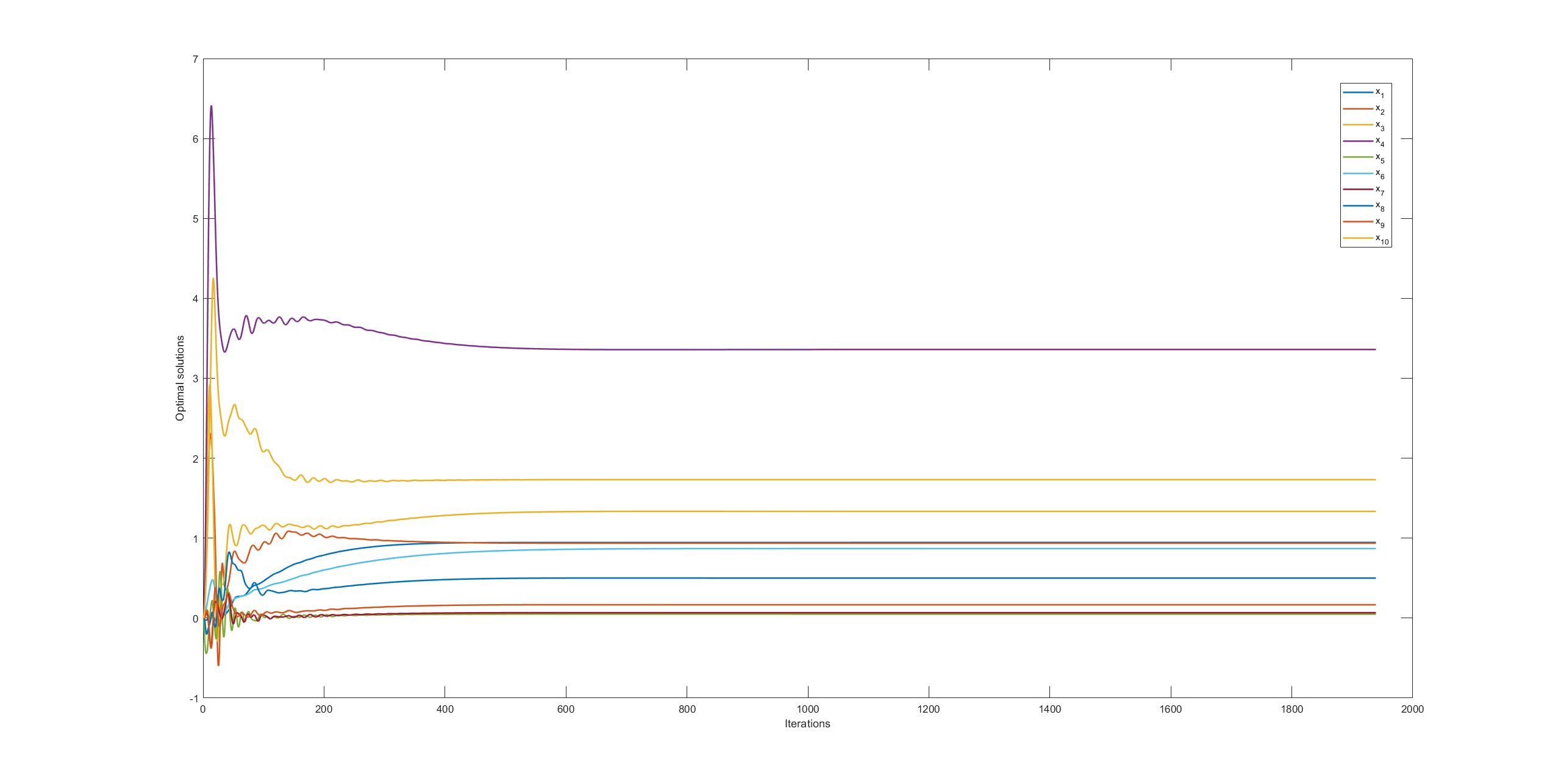}
			\caption{Optimal solution via parallel splitting}
			\label{fig:n10m10cvoptsolPS}
		\end{subfigure}
		\caption{Convergence of the optimal solution of problem $(D'_L)$ ($d=10$, $m=10$) via primal-dual interior point and parallel splitting methods, where each line represents one of the 10 components of the variables}
		\label{fig:n10m10cvoptsolEVaR-PDPS}
	\end{figure}
	\end{example}

\begin{example}
	We treat now the case where $d=m=15$. The best performances of our numerical simulations are illustrated in Tables \ref{tab:Entropyn15m15e1}--\ref{tab:Entropyn15m15e2}.
	
	\begin{table}[h!]
		\centering
		\begin{tabular}{lcc}
			\hline %\toprule
			Algorithm& \textbf{PS} ($\mu=1.83$, $\nu=2$)&\textbf{PD}\\
			\hline %\midrule
			Number of iterations & $1595$ & $634$ \\
			CPU time (seconds) & $3.4189$ & $5.702284$\\
			Objective value & $-2.8700$ & $-2.860610$ \\
			\hline %\bottomrule
		\end{tabular}
		\caption{Comparison of splitting (PS) and primal-dual interior point (PD) methods  when $d=15$ and $m=15$ for $\epsilon_1=10^{-3}$}
		\label{tab:Entropyn15m15e1}
	\end{table}
	
	\begin{table}[h!]
		\centering
		\begin{tabular}{lcc}
			\hline %\toprule
			Algorithm& \textbf{PS} ($\mu=1.83$, $\nu=2$)&\textbf{PD}\\
			\hline %\midrule
			Number of iterations & $2743$ & $634$ \\
			CPU time (seconds) & $5.7414$ & $5.5653$\\
			Objective value & $-2.861053$ & $-2.860610$ \\
			\hline %\bottomrule
		\end{tabular}
		\caption{Comparison of splitting (PS) and primal-dual interior point (PD) methods  when $d=15$ and $m=15$  for $\epsilon_2=10^{-4}$}
		\label{tab:Entropyn15m15e2}
	\end{table}
	
	%\begin{table}[h!]
	%	\centering
	%	\begin{tabular}{lcc}
	%		\hline %\toprule
	%		Algorithm&\textbf{PS} ($\mu=1.83$, $\nu=2$) &\textbf{PD} \\
	%		\hline %\midrule
	%		Number of iterations & $5432$ & $636$ \\
	%		CPU time (seconds) & $11.4630$ & $5.5811$ \\
	%		Objective value & $-2.860556$  & $-2.860555$ \\
	%		\hline %\bottomrule
	%	\end{tabular}
	%	\caption{Comparison of splitting (PS) and primal-dual interior point (PD) methods  when $d=15$ and $m=15$  for $\epsilon_3=10^{-6}$}
	%	\label{tab:Entropyn15m15e3}
	%\end{table}
	
	%\begin{table}[h!]
	%	\centering
	%	\begin{tabular}{lcc}
	%		\hline %\toprule
	%		Algorithm&\textbf{PS} ($\mu=1.83$, $\nu=2$) &\textbf{PD} \\
	%		\hline %\midrule
	%		Number of iterations & $8142$  & $639$ \\
	%		CPU time (seconds) & $17.9040$ & $5.5849$ \\
	%		Objective value & $-2.860552$  & $-2.860552$ \\
	%		\hline %\bottomrule
	%	\end{tabular}
	%	\caption{Comparison of splitting (PS) and primal-dual interior point (PD) methods  when $d=15$ and $m=15$  for $\epsilon_4=10^{-8}$}
	%	\label{tab:Entropyn15m15e4}
	%\end{table}
	
The parallel splitting method delivered for $(PU')$ the following optimal solution
$$	
\left\{
\begin{array}{l}
		\bar{x}=(0.5714, 2.6062, 1.3658, 3.9575, 0.0427, 2.4670, 0.1018, 0.1528,
		 2.3530, 0.1957, 0.8402,\\    0.6214, 0.1763, 0.6453, 0.3104)^\top\\
\bar{y}=0_{\mathbb{R}^{15}}\\		
\bar{\tau}= (-0.0125, 0.2369, 0.0472, 0.2945, -0.0054, 0.0136, -0.0238, -0.022, 0.0987, -0.0345,\\ -0.0118, -0.0217, -0.0176,-0.0277, -0.0024)^\top\\
		\bar{t}= 2.6233.
	\end{array}
\right.$$
	
Solving the optimality conditions system of equations provided in Theorem \ref{theoremChapter1:3.2}, we obtain for $(D'_L)$ the optimal solution
$$	
\left\{
\begin{array}{l}
\beta=5.4949\\
\bar{\lambda} = (0.0000, 0.0000, 0.7019, 0.0041, 0.2230, 0.0000, 0.0000, 0.0000, 0.0000, 0.0000, 0.0000,\\ 0.0000, 0.0000, 0.0710, 0.0000)^\top
	\end{array}
\right.$$

In this case, $v(PU')=v(D'_L)=-2.860552$.
	
	The primal-dual interior point method from \cite{AJF} needs again fewer iterations to deliver than the parallel splitting method, however, the latter converges faster when the accuracy is $\epsilon_2=10^{-3}$. However, for the tolerable error  $\epsilon_2=10^{-4}$, both algorithms converge in roughly the same time.

The convergence of objective function values and optimal solutions to the portfolio optimization problem $(D'_L)$ is illustrated in Figure
 \ref{fig:n15m15cvobjvalEVaR-PDPS} and Figure \ref{fig:n15m15cvoptsolEVaR-PDPS}, respectively.
	
	\begin{figure}
		\centering
		\begin{subfigure}{.45\textwidth}
	\centering
			\includegraphics[width=1\linewidth]{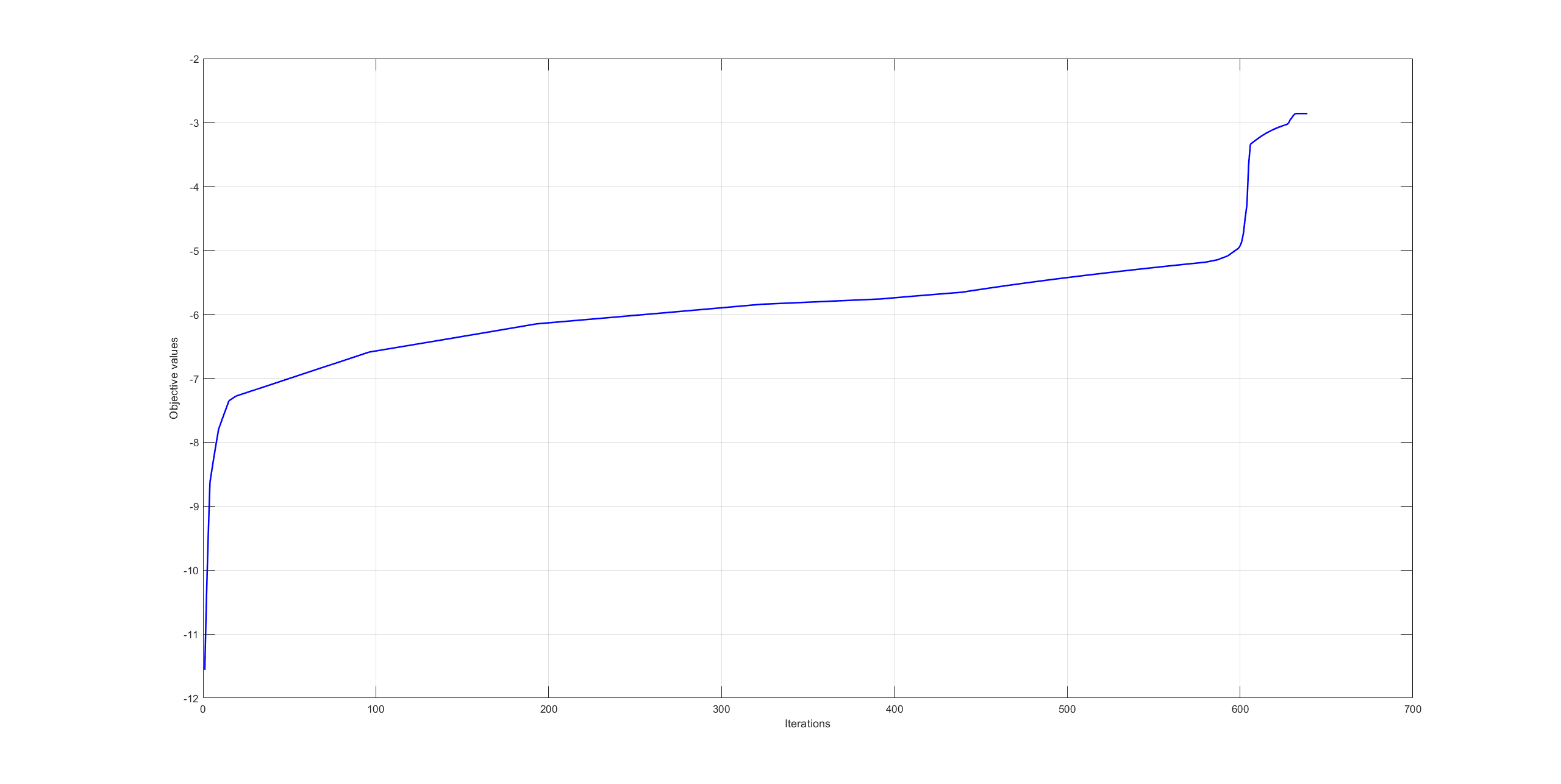}
			\caption{Objective value via primal-dual interior point}
			\label{fig:n15m15cvobjvalEVaR-PD}
		\end{subfigure}
		\begin{subfigure}{.45\textwidth}
			\centering
			\includegraphics[width=1\linewidth]{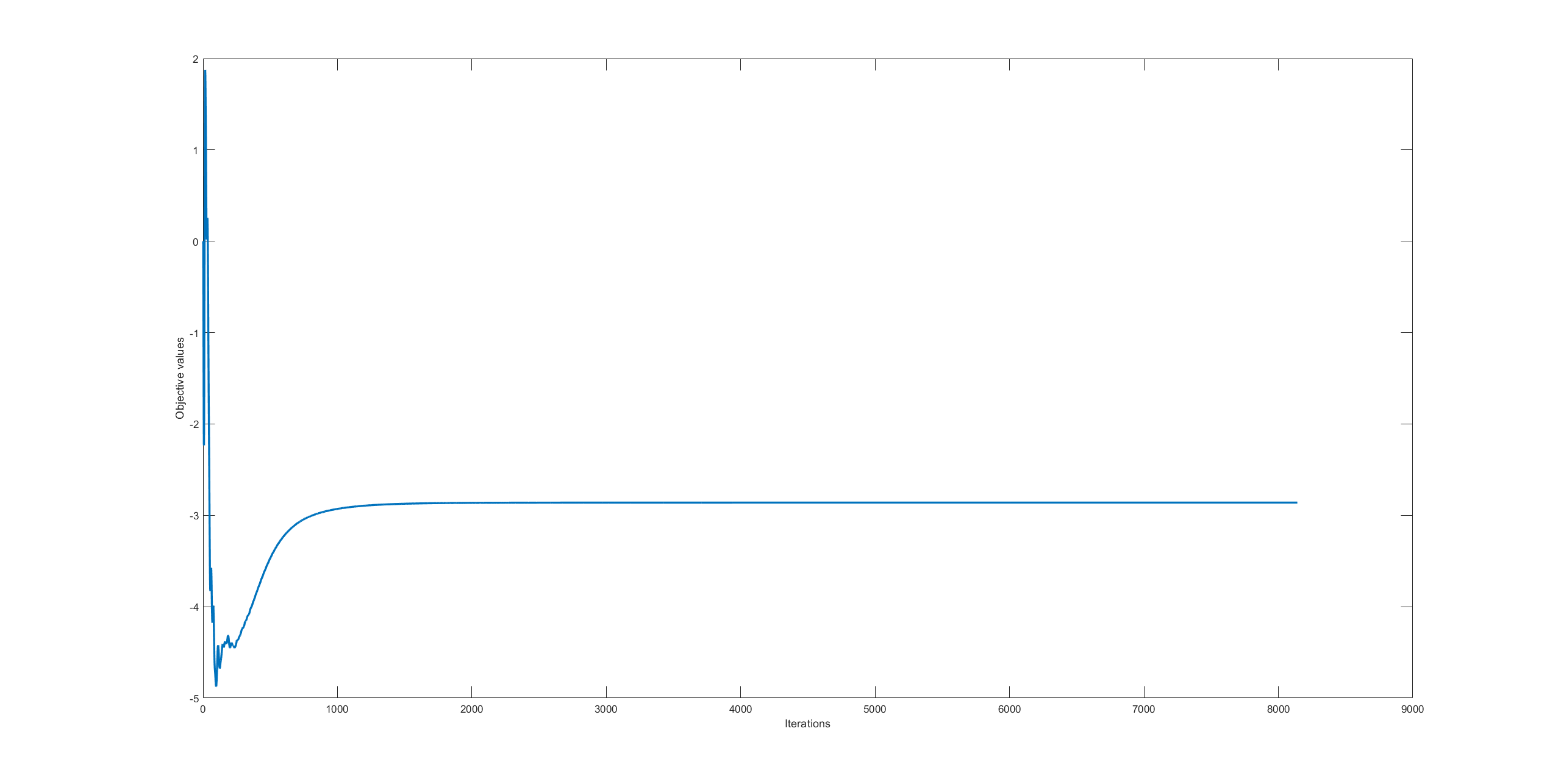}
			\caption{Objective value via parallel splitting}
			\label{fig:n15m15cvobjvalPS}
		\end{subfigure}
		\caption{Convergence of the objective value of problem $(D'_L)$ ($d=15$, $m=15$) via primal-dual interior point and parallel splitting methods}
		\label{fig:n15m15cvobjvalEVaR-PDPS}
	\end{figure}
	
	\begin{figure}
		\centering
		\begin{subfigure}{.45\textwidth}
			\centering
			\includegraphics[width=1\linewidth]{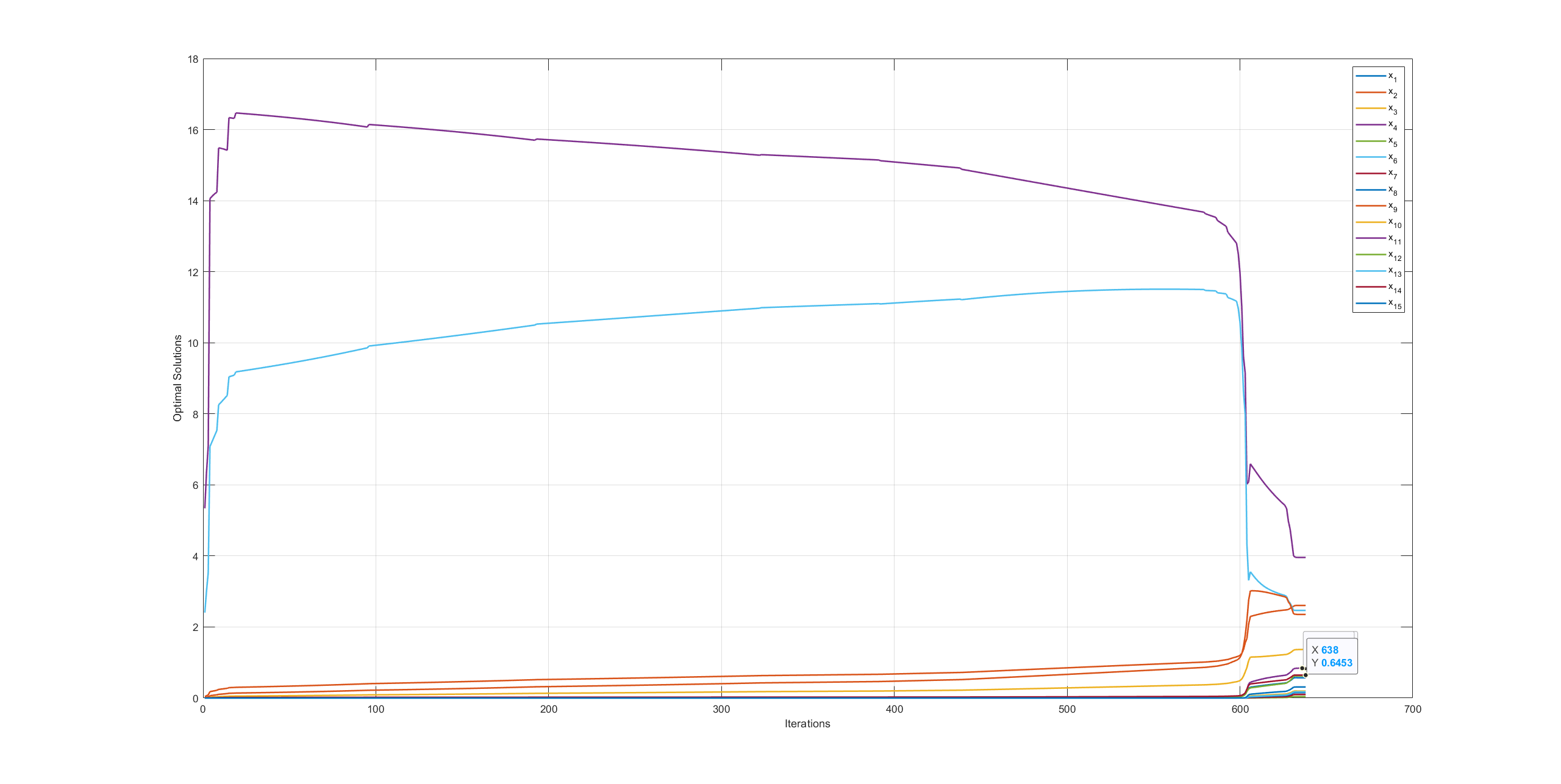}
			\caption{Optimal solution via primal-dual interior point}
			\label{fig:n15m15cvoptsolEVaR-PD}
		\end{subfigure}
		\begin{subfigure}{.45\textwidth}
			\centering
			\includegraphics[width=1\linewidth]{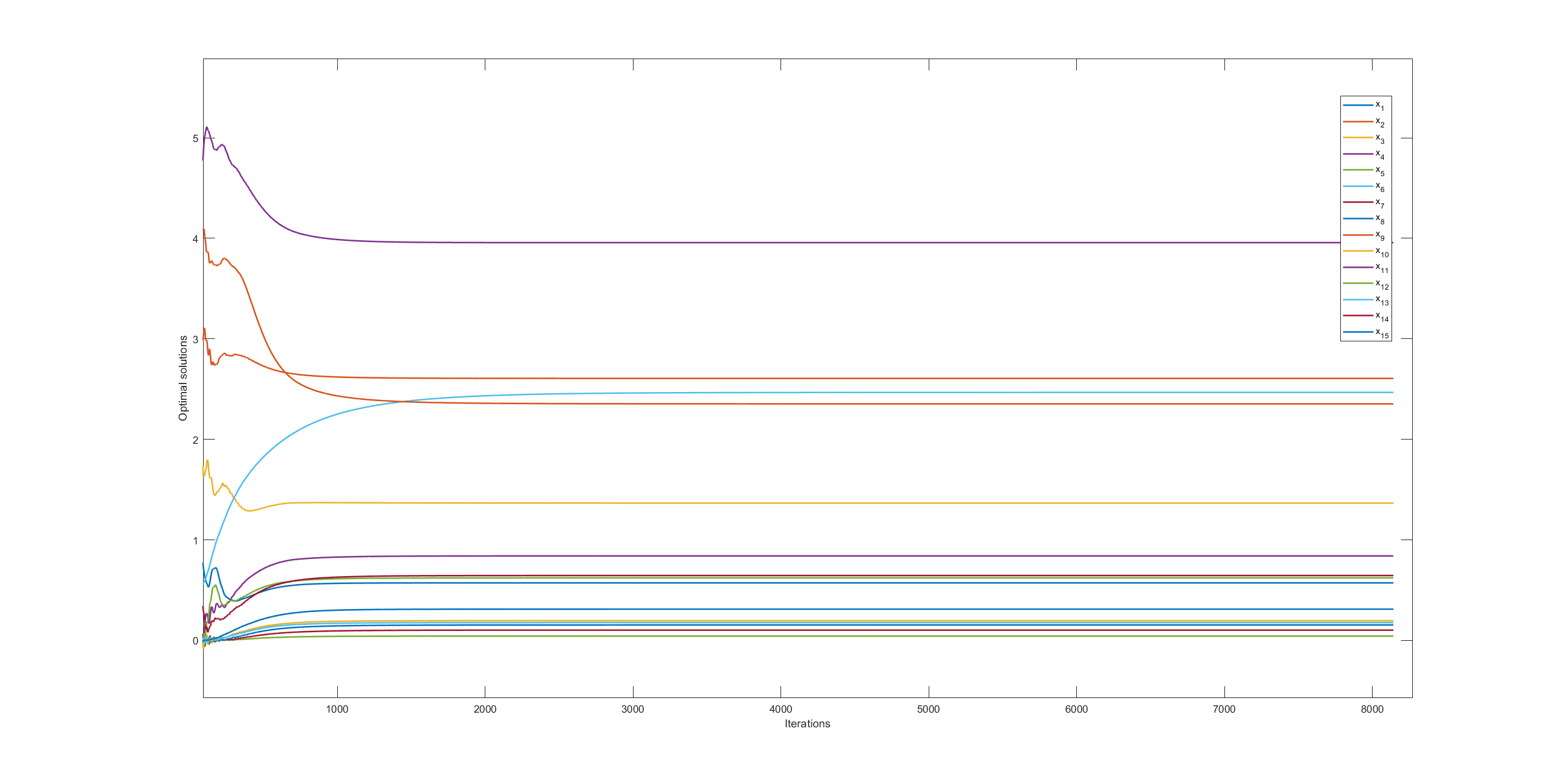}
			\caption{Optimal solution via parallel splitting}
			\label{fig:n15m15cvoptsolPS}
		\end{subfigure}
		\caption{Convergence of the optimal solution of problem $(D'_L)$ ($d=15$, $m=15$) via primal-dual interior point and parallel splitting methods, where each line represents one of the 15 components of the variables}
		\label{fig:n15m15cvoptsolEVaR-PDPS}
	\end{figure}
	
\end{example}

In the following two examples, we explore two different scenarios where the numbers of samples ($m$) and financial instruments ($n$) are not equal. The data for these tests, including the weights $p$, vectors $c$ and $h$ and the matrix $A$ were randomly generated using \textsc{MATLAB}. %This approach is the same as in other examples in this chapter on portfolio optimization, ensuring that the comparisons are consistent.

\begin{example}
	Let $d=10$ and $m=5$. The parallel splitting method delivered for $(D'_L)$ the optimal solution
	%\begin{align*}
	%	&\bar{x}=\left(2.8000, 0.1145, 0.7090, 0.6939, 0.3219\right)^\top,\\
	%	&\bar{\tau}= \left(0.7222,-0.0606,-0.0598,-0.0089,-0.0820\right)^\top,\\
	%	&\bar{t}=10.7293.
	%\end{align*}
	%As $D=\mathbb{R}^{10}$, we have $\bar{y}=0_{\mathbb{R}^{10}}$.\\
	%By solving linear system (\ref{S3}), the optimal asset allocation $\bar{\lambda}$ is given by (as well as

$$\left\{
\begin{array}{l}
\bar{\beta}=2.0589\\
\bar{\lambda} = \left(0.0000, 0.8763, 0.0000, 0.0000, 0.0000, 0.0000, 0.0000, 0.0000, 0.000, 0.1237\right),^\top
\end{array}\right.$$
while the optimal objective value is $v(PU')=v(D'_L)=9.667241$.
	
	\begin{table}[h!]
		\centering
		\begin{tabular}{lcc}
			\hline %\toprule
			Algorithm&\textbf{PS} ($\mu=1.45$, $\nu=2.2$)&\textbf{PD} \\
			\hline %\midrule
			Number of iterations & $353$ & $470$ \\
			CPU time (seconds) & $0.3906$ & $4.3591$ \\
			Objective value & $9.667381$ &$9.666822$\\
			\hline %\bottomrule
		\end{tabular}
		\caption{Comparison of splitting (PS) and primal-dual interior point (PD) methods when $d=10$ and $m=5$ for $\epsilon_1=10^{-3}$}
		\label{tab:Entropyn10m5e1}
	\end{table}
	
	\begin{table}[h!]
		\centering
		\begin{tabular}{lcc}
			\hline %\toprule
			Algorithm&\textbf{PS} ($\mu=1.45$, $\nu=2.2$)&\textbf{PD} \\
			\hline %\midrule
			Number of iterations & $460$ & $471$ \\
			CPU time (seconds) & $0.5383$ & $4.1449$ \\
			Objective value & $9.667280$ &$9.667154$\\
			\hline %\bottomrule
		\end{tabular}
		\caption{Comparison of splitting (PS) and primal-dual interior point (PD) methods when $d=10$ and $m=5$ for $\epsilon_2=10^{-4}$}
		\label{tab:Entropyn10m5e2}
	\end{table}
	
	\begin{table}[h!]
		\centering
		\begin{tabular}{lcc}
			\hline %\toprule
			Algorithm&\textbf{PS} ($\mu=1.45$, $\nu=2.2$)&\textbf{PD} \\
			\hline %\midrule
			Number of iterations &$673$ & $474$ \\
			CPU time (seconds) & $0.7169$ & $4.1813$ \\
			Objective value& $9.667242$ & $9.667240$ \\
			\hline %\bottomrule
		\end{tabular}	
		\caption{Comparison of splitting (PS) and primal-dual interior point (PD) methods when $d=10$ and $m=5$ for  $\epsilon_3=10^{-6}$}\label{tab:Entropyn10m5e3}
	\end{table}
	
	\begin{table}[h!]
		\centering
		\begin{tabular}{lcc}
			\hline %\toprule
			Algorithm&\textbf{PS} ($\mu=1.45$, $\nu=2.2$)&\textbf{PD} \\
			\hline %\midrule
			Number of iterations &$886$ & $476$ \\
			CPU time (seconds) & $0.8139$ & $4.4157$ \\
			Objective value& $9.667241$ & $9.667241$ \\
			\hline %\bottomrule
		\end{tabular}	
		\caption{Comparison of splitting (PS) and primal-dual interior point (PD) methods when $d=10$ and $m=5$ for $\epsilon_4=10^{-8}$}\label{tab:Entropyn10m5e4}
	\end{table}
	
The results displayed in Tables \ref{tab:Entropyn10m5e1}--\ref{tab:Entropyn10m5e4} highlight that although the primal-dual interior point method typically requires fewer iterations than the parallel splitting method, this is not the case in this example. Notably, for error tolerances $\epsilon_1=10^{-3}$ and $\epsilon_2=10^{-4}$, our parallel splitting needs only $353$ iterations to deliver an optimal solution to the considered portfolio optimization problem, compared to the $470$ needed by the EVaR-PD method for $\epsilon_1$ (also $460$ versus $471$ for $\epsilon_2$). Furthermore, the parallel splitting method demonstrates faster computational times for all error tolerances. Specifically, our method that involves a parallel splitting algorithm converged 10 times faster than the primal-dual interior point method from \cite{AJF} for $\epsilon_1=10^{-3}$ and almost 9 as fast for $\epsilon_4=10^{-8}$, while both methods achieved comparable objective values.
\end{example}

\begin{remark}\normalfont
The graphs in Figure \ref{fig:n10m10cvobjvalEVaR-PDPS}$(b)$ and Figure \ref{fig:n15m15cvobjvalEVaR-PDPS}$(b)$ look somehow counterintuitive, as the objective values of the portfolio optimization problems (in this case maximization problems) abruptly decrease from very large values and then slightly increase toward the (approximate) optimal values of the problems, different from the expected monotonicity (induced by proximal point methods). These graphs are obtained by solving at each iteration of the parallel splitting method for solving $(PU')$ the optimality conditions system from Theorem \ref{theoremChapter1:3.2} and then plotting the corresponding dual objective values. While the inherent computational approximations of the employed computer certainly play a role in this phenomenon, its main cause seems to be the lack of dual feasibility of the derived dual solutions in the first iterations. One can see, for instance, in Figure \ref{fig:n10m10cvoptsolEVaR-PDPS} $(b)$ and Figure \ref{fig:n15m15cvoptsolEVaR-PDPS}$(b)$ that the variable represented in pink is not feasible from the beginning. However, this flaw is corrected as the algorithm advances and the (approximate) optimal values of the problems obtained via the two competing methods are (approximatively) the same.
\end{remark}

\section{Conclusions}

We propose a new approach for solving convex portfolio optimization problems formulated by means of the entropic value-at-risk that makes use of Lagrange duality and is based on splitting proximal point methods. More precisely, we formulated a convex optimization problem that can be solved via a parallel splitting proximal point algorithm and whose Lagrange dual problem is our target, the portfolio optimization problem. In order to be able to employ the parallel splitting proximal point algorithm, we provided a new epigraphical projection formula, for the projection on the epigraph of the Shannon entropy function. After solving the considered problem with a parallel splitting proximal point algorithm, optimal solutions to the convex portfolio optimization problems formulated by means of the entropic value-at-risk are recovered by solving a system of equations that corresponds to the necessary and sufficient optimality conditions for the mentioned primal-dual pair of optimization problems. Comparisons of our computational results with the primal-dual interior point method proposed in \cite{AJF} to the same end reveal that our approach is competitive against state of the art methods.\\

{\bf Conflicts of interests.} The authors have no conflicts of interests to report.\\

{\bf Acknowledgements.} The work of Nail Baloul and Sorin-Mihai Grad was partially supported by the FMJH Program Gaspard Monge for optimization and operations research and their interactions with data science. The author Oleg Wilfer conducted this research while affiliated with Faculty of Mathematics, Chemnitz University of Technology.

\end{document}